\documentclass[11pt]{article}

\usepackage{amssymb,amsmath,amsthm,epsfig,bbold,amscd,thmtools,mathtools}
\usepackage{tikz}

\numberwithin{equation}{section}
\theoremstyle{plain}
\newtheorem{thm}{Theorem}[section]
\newtheorem{lem}{Lemma}[section]

\theoremstyle{definition}
\newtheorem{defn}{Definition}

\newtheorem{rem}{Remark}[section]

\numberwithin{defn}{section}

\begin{document}

\title{Pointwise Bounds and Blow-up for\\ Nonlinear Fractional Parabolic Inequalities}

\author{Steven D. Taliaferro\\
Department of Mathematics\\
Texas A\&M University\\
College Station, TX 77843-3368\\
USA\\
{\tt stalia@math.tamu.edu}}

\date{}
\maketitle	

\begin{abstract}
We investigate pointwise upper bounds for nonnegative
solutions $u(x,t)$ of the nonlinear initial value problem
\begin{equation}\label{0.1}
 0\leq(\partial_t-\Delta)^\alpha u\leq u^\lambda 
\quad\text{ in }\mathbb{R}^n \times\mathbb{R},\,n\geq1,
\end{equation}
\begin{equation}\label{0.2}
 u=0\quad\text{in }\mathbb{R}^n\times(-\infty,0)
\end{equation}
where $\lambda$ and $\alpha$ are positive constants.  To do this we
first give a definition---tailored for our study of \eqref{0.1},
\eqref{0.2}---of  fractional powers of the heat operator
$(\partial_t-\Delta)^\alpha :Y\to X$ where $X$ and
$Y$ are linear spaces whose elements are real valued functions on
$\mathbb{R}^n \times\mathbb{R}$ and $0<\alpha<\alpha_0$ for some
$\alpha_0$ which depends on $n$, $X$ and $Y$.

We then obtain, when they exist, optimal pointwise upper bounds on
$\mathbb{R}^n \times(0,\infty)$ for nonnegative solutions $u\in Y$ of
the initial value problem \eqref{0.1}, \eqref{0.2} with particular
emphasis on those bounds as $t\to0^+$ and as $t\to\infty$.
\medskip 


\end{abstract}

\section{Introduction}\label{sec1}
In this paper we study pointwise upper bounds for nonnegative
solutions $u(x,t)$ of the nonlinear inequalities
\begin{equation}\label{1.2}
 0\leq(\partial_t-\Delta)^\alpha u\leq u^\lambda \quad\text{in }\mathbb{R}^n \times\mathbb{R},\,n\geq1,
\end{equation}
satisfying the initial condition
\begin{equation}\label{1.3}
 u=0\quad \text{in }\mathbb{R}^n \times(-\infty,0)
\end{equation}
where $\lambda$ and $\alpha$ are positive constants.

To do this, we first give in Section \ref{sec2} a definition---appropriate for our analysis of the
initial value problem \eqref{1.2}, \eqref{1.3}---of fractional powers of
the heat operator
\begin{equation}\label{1.1}
 (\partial_t-\Delta)^\alpha :Y\to X
\end{equation}
where $\Delta$ is the Laplacian with respect to $x\in\mathbb{R}^n$,
$X$ and $Y$ are linear spaces whose elements are real valued functions
on $\mathbb{R}^n \times\mathbb{R}$, and $0<\alpha<\alpha_0$ for some
$\alpha_0 >0$ which depends on $n$, $X$ and $Y$.

With the definition of \eqref{1.1} in hand, we obtain, when they exist,
optimal pointwise upper bounds on $\mathbb{R}^n\times(0,\infty)$ for
nonnegative solutions $u\in Y$ of the initial value problem
\eqref{1.2}, \eqref{1.3} with particular emphasis on these bounds as
$t\to0^+$ and as $t\to\infty$. These results are stated in Section
\ref{sec3} and proved in Section \ref{sec8}.

Since the operator \eqref{1.1} is nonlocal, we must require the
initial condition \eqref{1.3} to hold in $\mathbb{R}^n
\times(-\infty,0)$ (not just in $\mathbb{R}^n
\times \{0\}$) and nonnegative solutions of \eqref{1.2}, \eqref{1.3} may
not tend pointwise to zero as $t\to 0^+$ (see Theorem \ref{thm3.5}) even
though they satisfy the initial condition \eqref{1.3}. 

Of course any estimates we obtain for nonnegative solutions of
\eqref{1.2}, \eqref{1.3} also hold for nonnegative solutions of
the initial value problem consisting of
\eqref{1.3} and the {\it equation}
\[
 (\partial_t-\Delta)^\alpha u=u^\lambda \quad \text{in }\mathbb{R}^n \times\mathbb{R}.
\]

According to our results in Section \ref{sec3} there are essentially
only three possibilities for the solutions of \eqref{1.2},
\eqref{1.3} depending on $X$, $Y$, $\lambda$, and $\alpha$:

\begin{enumerate}
\item The only solution is $u\equiv 0$ in $\mathbb{R}^n\times\mathbb{R}$;
\item There exist sharp nonzero pointwise bounds for solutions as
  $t\to 0^+$ and as $t\to\infty$;
\item There do not exist pointwise bounds for solutions as
  $t\to 0^+$ and as $t\to\infty$.
\end{enumerate}
All possiblities can occur. For the precise statements of
possibilities (i), (ii), and (iii) see Theorem \ref{thm3.1}, Theorems
\ref{thm3.2}--\ref{thm3.4}, and Theorems \ref{thm3.5} and
\ref{thm3.6}, respectively.

The operator  \eqref{1.1} is a fully fractional heat operator as
opposed to time fractional heat operators in which the fractional
derivatives are only with respect to $t$, and space fractional heat
operators, in which the fractional derivatives are only with respect to $x$.

Some recent results for nonlinear PDEs containing time (resp. space)
fractional heat operators can be found in \cite{AV,A,ACV,DVV,K,
  KSVZ,M,OD,SS,VZ,ZS}
(resp. \cite{AABP,AMPP,BV,CVW,DS,FKRT,GW,JS,MT,PV,V,VV,VPQR}). We know
of no results for nonlinear PDEs containing the fully fractional heat
operator  \eqref{1.1}. However results for linear PDEs containing
\eqref{1.1}, including in particular 
\[
(\partial_t-\Delta)^\alpha u=f,
\]
where $f$ is a given function, can be found in \cite{ACM,NS,SK,ST}.

\section{Definition and properties of  fully fractional heat operators}\label{sec2}
In this section we give a well-motivated definition of the fully fractional heat operator \eqref{1.1}, suitable for our study of the initial value problem \eqref{1.2}, \eqref{1.3}, and then give some of its properties.

Some of the material in this section is inspired by---and can be viewed
as the parabolic analog of---the material in \cite[Sec. 5.1]{S}
concerning the fractional Laplacian.

Since for functions $u:\mathbb{R}^n \times\mathbb{R}\to\mathbb{R},\,n\geq1$, which are sufficiently smooth and small at infinity we have
$$((\partial_t-\Delta)u)\ \widehat{ }\ (y,s)=(|y|^2 -is)\widehat{u}(y,s),$$
where $\ \widehat{ }\ $ is the Fourier transform operator on $\mathbb{R}^n \times\mathbb{R}$ given by 
$$\widehat{u}(y,s)=\iint_{\mathbb{R}^n \times\mathbb{R}}e^{i(y,s)\cdot(x,t)}u(x,t)\, dx \, dt,$$
the fractional heat operator $(\partial_t-\Delta)^\alpha ,\,\alpha>0$, is formally defined in \cite[Chapter 2]{SP} by
\begin{equation}\label{2.1}
 ((\partial_t-\Delta u)^\alpha u)\ \widehat{ }\ (y,s)=(|y|^2 -is)^\alpha \widehat{u}(y,s).
\end{equation}
If $f=(\partial_t-\Delta)^\alpha u$ then from \eqref{2.1} and the fact 
(see \cite[Theorem 2.2]{SP} and Theorem \ref{thm2.1}(i) below) that
$$\widehat{\Phi}_\alpha(y,s)=(|y|^2 -is)^{-\alpha}\quad\text{for }0<\alpha<(n+2)/2$$
in the sense of tempered distributions where
\begin{equation}\label{2.2}
 \Phi_\alpha (x,t)=\frac{t^{\alpha-1}}{\Gamma(\alpha)}\,
\frac{1}{(4\pi t)^{n/2}}e^{-|x|^2/(4t)}\raisebox{2pt}{$\chi$}_{(0,\infty)}(t),
\end{equation}
we formally get
$$\widehat{u}=\widehat{\Phi}_\alpha \widehat{f}.$$
Hence by the convolution theorem we formally find that
\begin{equation}\label{2.3}
 u=J_\alpha f:=\Phi_\alpha *f
\end{equation}
where $*$ is the convolution operation in $\mathbb{R}^n \times\mathbb{R}$.  Since $\Phi_\alpha (x,t)=0$ for $t\leq0$ we have
\begin{equation}\label{2.4}
 J_\alpha f(x,t)=\iint_{\mathbb{R}^n \times(-\infty,t)}\Phi_\alpha (x-\xi,t-\tau)f(\xi,\tau)\, d\xi \, d\tau .
\end{equation}

By part (ii) of the following theorem, equations \eqref{2.1} and
\eqref{2.3} are equivalent in the sense that 
\[
(J_\alpha f)\ \widehat{ }\ =(|y|^2 -is)^{-\alpha} \widehat{f}\quad\text{for }f\in
L^1 (\mathbb{R}^n \times\mathbb{R}) \text{ and } 0<\alpha<(n+2)/2
\]
in the sense of
tempered distributions.

\begin{thm}\label{thm2.1}
 Suppose $0<\alpha<(n+2)/2$.
 \begin{enumerate}
 \item[(i)] The Fourier transform of $\Phi_\alpha (x,t)$ is the
   function $(|y|^2 -is)^{-\alpha}$ in the sense that
  $$\iint_{\mathbb{R}^n \times\mathbb{R}}\Phi_\alpha (x,t)\widehat{\varphi}(x,t)\, dx \, dt=\iint_{\mathbb{R}^n \times\mathbb{R}}(|y|^2-is)^{-\alpha}\varphi(y,s)\, dy \, ds$$
  for all $\varphi\in S$ where $S$ is the Schwarz class of rapidly decreasing functions.
  \item[(ii)] The identity $(J_\alpha f)\ \widehat{ }\  (y,t)=(|y|^2 -is)^{-\alpha}\widehat{f}(y,s)$ holds in the sense that
  \begin{equation}\label{2.5}
   \iint_{\mathbb{R}^n \times\mathbb{R}}J_\alpha f(x,t)\widehat{g}(x,t)\, dx \, dt=\iint_{\mathbb{R}^n \times\mathbb{R}}(|y|^2 -is)^{-\alpha}\widehat{f}(y,s)g(y,s)\, dy \, ds
  \end{equation}
  for all $f\in L^1 (\mathbb{R}^n \times\mathbb{R})$ and all $g\in S$.
 \end{enumerate}
\end{thm}

Motivated by these formal calculations, we will now define the operator $(\partial_t-\Delta)^\alpha$ as the inverse of a linear operator
\begin{equation}\label{2.6}
 J_\alpha :X\to Y
\end{equation}
where $J_\alpha$ is defined by \eqref{2.4} and \eqref{2.2} and $X$ and
$Y$ are linear spaces whose elements are functions
$f:\mathbb{R}^n \times\mathbb{R}\to\mathbb{R}$ such that the operator
\eqref{2.6} has the following properties:
\begin{enumerate}
 \item[(P1)] it makes sense because the integral in \eqref{2.4}
   defines a real valued measurable function on
   $\mathbb{R}^n\times\mathbb{R}$ for all $f\in X$, 
 \item[(P2)] it is one-to-one and onto, and
 \item[(P3)] if $u=J_\alpha f$ then $f=0$ in $\mathbb{R}^n \times(-\infty,0)$ if and only if $u=0$ in $\mathbb{R}^n \times(-\infty,0)$.
\end{enumerate}
Property (P3) will be needed to handle the initial condition \eqref{1.3}.  The domain of $J_\alpha$ is usually taken to be $L^p (\mathbb{R}^n \times\mathbb{R}),\,1\leq p<\frac{n+2}{2\alpha}$ (see \cite[Section 9.2]{SK}).  However since the region of integration
for the integral \eqref{2.4} is not $\mathbb{R}^n \times\mathbb{R}$ but rather $\mathbb{R}^n \times(-\infty,t)$, we see that more natural and less restrictive choices for the domain and range of $J_\alpha$ are
\begin{align}\label{2.9}
 X^p &:=\bigcap_{T\in\mathbb{R}}L^p (\mathbb{R}^n \times\mathbb{R}_T )\\
\label{2.10}
 Y^p_\alpha &:=J_\alpha (X^p )
\end{align}
respectively, where $\mathbb{R}_T =(-\infty,T)$.  By \eqref{2.9} we
mean $X^p$ is the set of all measurable functions
$f:\mathbb{R}^n \times\mathbb{R}\to\mathbb{R}$ such that
$$\| f\|_{L^p (\mathbb{R}^n \times\mathbb{R}_T )}<\infty\quad\text{for all }T\in\mathbb{R}.$$
The notation in \eqref{2.9} should be interpreted similarly elsewhere in this paper.

According to the following two theorems the formal operator
\begin{equation}\label{2.11}
 J_\alpha :X^p \to Y^p_\alpha ,
\end{equation}
where $X^p$ and $Y^p_\alpha$ are defined in \eqref{2.9} and \eqref{2.10}, satisfies properties (P1)--(P3) provided either 
\begin{equation}\label{2.12}
 \left(p>1 \text{ and }  0<\alpha<\frac{n+2}{2p}\right)\quad\text{or}
\quad\left(p=1 \text{ and }  0<\alpha\leq\frac{n+2}{2p}\right).
\end{equation}

When $p$ and $\alpha$ satisfy \eqref{2.12}, part (i) of the following theorem shows that the operator \eqref{2.11} satisfies (P1) and parts (ii) and (iii) give some of its properties.

\begin{thm}\label{thm2.2}
 Suppose $p$ and $\alpha$ are real numbers satisfying \eqref{2.12} and $f\in X^p$.  Then
 \begin{enumerate}
  \item[(i)] $J_\alpha f,\,J_\alpha |f|\in
    L^{p}_{\text{loc}}(\mathbb{R}^n \times\mathbb{R})$ and
  \item[(ii)] $J_\beta (J_\gamma f)=J_\alpha f$ in $L^{p}_{\text{\rm loc}}
    (\mathbb{R}^n \times\mathbb{R})$ whenever $\beta>0,\,\gamma>0$,
    and $\beta+\gamma=\alpha$.
\end{enumerate}
If in addition $\alpha>1$ then
\begin{enumerate}  
\item[(iii)] $HJ_\alpha f=J_{\alpha-1}f$ in $\mathcal{D}^\prime (\mathbb{R}^n \times\mathbb{R})$ where $H=\partial_t-\Delta$ is the heat operator.
\end{enumerate}
\end{thm}

\begin{rem}\label{rem2.1}
 Theorem \ref{thm2.2}(i) can be improved to $J_\alpha f\in L^{q}_{\text{loc}}(\mathbb{R}^n \times\mathbb{R})$ when 
 $$1<p<\frac{n+2}{2\alpha} \quad\text{ and }\quad  \frac{1}{q}=\frac{1}{p}-\frac{2\alpha}{n+2}.$$
 This can be seen by applying Gopala Rao \cite[Theorem 3.1]{GR} to the function $f_T$ defined in the proof of Theorem \ref{thm2.2} in Section \ref{sec6}.
\end{rem}

According to the following theorem, if $p$ and $\alpha$ satisfy \eqref{2.12} then the operator \eqref{2.11} satisfies properties (P2) and (P3) where $X^p$ and $Y^p_\alpha$ are defined by \eqref{2.9} and \eqref{2.10}.

\begin{thm}\label{thm2.3}
 Suppose $p$ and $\alpha$ are real numbers satisfying \eqref{2.12}.  Then
 \begin{enumerate}
  \item[(i)] the operator \eqref{2.11} is one-to-one and onto, and
  \item[(ii)] if
  \begin{equation}\label{2.13}
   f\in X^p  \text{ and }  T\in\mathbb{R}
  \end{equation}
  then
  $$f|_{\mathbb{R}^n \times\mathbb{R}_T} =0\quad\text{if and only if}
\quad (J_\alpha f)|_{\mathbb{R}^n \times\mathbb{R}_T}=0.$$
 \end{enumerate}
\end{thm}

By the results in this section, the following definition is natural and makes sense.
\begin{defn}\label{def2.1}
 Suppose $p$ and $\alpha$ are real numbers satisfying \eqref{2.12} and $X^p$ and $Y^p_\alpha$ are defined by \eqref{2.9} and \eqref{2.10}.  Then the operator
 \begin{equation}\label{2.14}
  (\partial_t-\Delta)^\alpha :Y^p_\alpha \to X^p
 \end{equation}
 is defined to be the inverse of the operator \eqref{2.11}.
\end{defn}

\begin{rem}\label{rem2.2}
The functions $\mu_T:X^p\to\mathbb{R}$, $T\in\mathbb{R}$, defined by
$\mu_T(f)=\|f\|_{L^p(\mathbb{R}^n \times\mathbb{R}_T)}$, form a
separating family of seminorms on $X^p$ which turns $X^p$ into a
locally convex topological vector space (see for example \cite[Theorem
1.37]{R}). Thus assuming \eqref{2.12} and defining a subset $O'$ of
$Y^p_\alpha$ to be open if $O'=J_\alpha(O)$ for some open set $O\in
X^p$, we see by Theorem \ref{2.3}(i) that $Y^p_\alpha$ is also a
locally convex topological vector space and the operator \eqref{2.14}
is a homeomorphism.
\end{rem}

We conclude this section by investigating
\[\lim_{a\to0^+}(\partial_t -a^2 \Delta)^\alpha \quad\text{ and }\quad\lim_{b\to0^+}(b\partial_t -\Delta)^\alpha\]
where $\alpha>0$.

To do this we first repeat the above procedure with
$\partial_t -\Delta$ replaced with $b\partial_t -a^2 \Delta$ where $a$
and $b$ are positive constants.  The end result after defining
\begin{equation}\label{J1}
 J_{\alpha,a,b}:X^p \to Y^{p}_{\alpha,a,b}:=J_{\alpha,a,b}(X^p )
\end{equation}
by
\[J_{\alpha,a,b}f=\Phi_{\alpha,a,b}*f,\]
where $a,b,\alpha,p$ are positive constants satisfying \eqref{2.12} and
$$\Phi_{\alpha,a,b}(x,t)=\frac{1}{a^n b}\Phi_\alpha \left(\frac{x}{a},\frac{t}{b}\right),$$
is the following modified version of Definition 2.1.

\begin{defn}
  Suppose $a,b,p$ and $\alpha$ are positive constants satsfying
  \eqref{2.12} and $X^p$ 
and $Y^{p}_{\alpha,a,b}$ are
defined in \eqref{2.9} and \eqref{J1}.  Then the operator
 $$(b\partial_t -a^2 \Delta)^\alpha :Y^{p}_{\alpha,a,b}\to X^p$$
 is defined to be the inverse of the operator \eqref{J1}.
\end{defn}

The following theorem states in what sense 
\[(\partial_t -a^2 \Delta)^\alpha \to\partial^{\alpha}_{t}\quad\text{ as }a\to0^+\]
where we formally define the equation
\[\partial^{\alpha}_{t}u=f\]
to mean
\[u=J_{\alpha,0,1}f\]
where
\[(J_{\alpha,0,1}f)(x,t):=\int^{t}_{-\infty}\frac{(t-\tau)^{\alpha-1}}{\Gamma(\alpha)}f(x,\tau)\,d\tau\]
is the Riemann-Liouville integral of $f$ with respect to $t$ of order
$\alpha$ with base point $-\infty$.

\begin{thm}\label{thm2.4}
  Suppose $\alpha>0$ and
  $f:\mathbb{R}^n \times\mathbb{R}\to\mathbb{R}$ is a continuous
  function with compact support.  Then
 \[J_{\alpha,a,1}f\to J_{\alpha,0,1}f\quad\text{ as }a\to0^+\]
 uniformly on compact subsets of $\mathbb{R}^n \times\mathbb{R}$.
\end{thm}

The following theorem states in what sense
$$(b\partial_t -\Delta)^\alpha \to(-\Delta)^\alpha \quad\text{ as }b\to0^+$$
where we formally define the equation
$$(-\Delta)^\alpha u=f$$
to mean
$$u=J_{\alpha,1,0}f$$
where
\[(J_{\alpha,1,0}f)(x,t):=\frac{1}{\gamma(n,\alpha)}\int_{\mathbb{R}^n}\frac{f(y,t)\,dy}{|x-y|^{n-2\alpha}}\]
is the Riesz potential of $f$ with respect to $x$ of order $\alpha$.  Here
\begin{equation}\label{R0}
 \gamma(n,\alpha)=\frac{4^\alpha \pi^{n/2}\Gamma(\alpha)}{\Gamma(n/2-\alpha)}.
\end{equation}

\begin{thm}\label{thm2.5}
  Suppose $0<2\alpha<n$ and
  $f:\mathbb{R}^n \times\mathbb{R}\to\mathbb{R}$ is a continuous
  function with compact support.  Then
 $$J_{\alpha,1,b}f\to J_{\alpha,1,0}f\quad\text{ as }b\to0^+$$
 uniformly on compact subsets of $\mathbb{R}^n \times\mathbb{R}$.
\end{thm}

\section{Results for fully fractional initial value problems}\label{sec3}
In this section we state our results concerning pointwise bounds for nonnegative solutions
\begin{equation}\label{3.1}
 u\in Y^p_\alpha
\end{equation}
of the fully fractional initial value problem
\begin{equation}\label{3.2}
 0\leq(\partial_t-\Delta)^\alpha u\leq u^\lambda \quad\text{in } \mathbb{R}^n \times\mathbb{R},\,n\geq1,
\end{equation}
\begin{equation}\label{3.3}
 u=0\quad\text{in } \mathbb{R}^n \times(-\infty,0)
\end{equation}
where $\lambda>0$ and, as in the Definition \ref{def2.1} of the operator \eqref{2.14}, $\alpha$ and $p$ satisfy \eqref{2.12}.

\begin{rem}\label{rem3.1}
 If $\alpha$ and $p$ satisfy \eqref{2.12} and $u$ satisfies \eqref{3.1} and the first inequality in \eqref{3.2} then
 $$f:=(\partial_t-\Delta)^\alpha u\geq0\quad\text{in } \mathbb{R}^n \times\mathbb{R}$$
 and hence $u=J_\alpha f\geq0$ in $\mathbb{R}^n \times\mathbb{R}$ by
 \eqref{2.4}.  Thus the assumption that $u$ be nonnegative can be
 omitted when studying \eqref{3.1}--\eqref{3.3}.
\end{rem}

In order to state our results we first note that for each fixed
$p\geq1$ the open first quadrant of the $\lambda\alpha$-plane is the
union of the following pairwise disjoint sets.
\begin{align*}
 &A=\left\{(\lambda,\alpha):\lambda\geq1 \text{ and }  \alpha>\dfrac{n+2}{2p}\left(1-\dfrac{1}{\lambda}\right)\right\}\\
 &B=\left\{(\lambda,\alpha):0<\lambda<1 \text{ and }  \alpha>0\right\}\\
 &C=\left\{(\lambda,\alpha):\lambda>1 \text{ and }  0<\alpha<\dfrac{n+2}{2p}\left(1-\dfrac{1}{\lambda}\right)\right\}\\
 &D=\left\{(\lambda,\alpha):\lambda>1 \text{ and }  \alpha=\dfrac{n+2}{2p}\left(1-\dfrac{1}{\lambda}\right)\right\}.
\end{align*}

Note that $A$, $B$, and $C$ are two dimensional regions in the
$\lambda\alpha$-plane whereas $D$ is the curve separating $A$ and $C$.
(See Figure \ref{fig1}.) Our results in this section deal with
solutions of \eqref{3.1}--\eqref{3.3} when $(\lambda,\alpha)$ is in
$A$, $B$, or $C$. We have no results when $(\lambda,\alpha)\in D$.

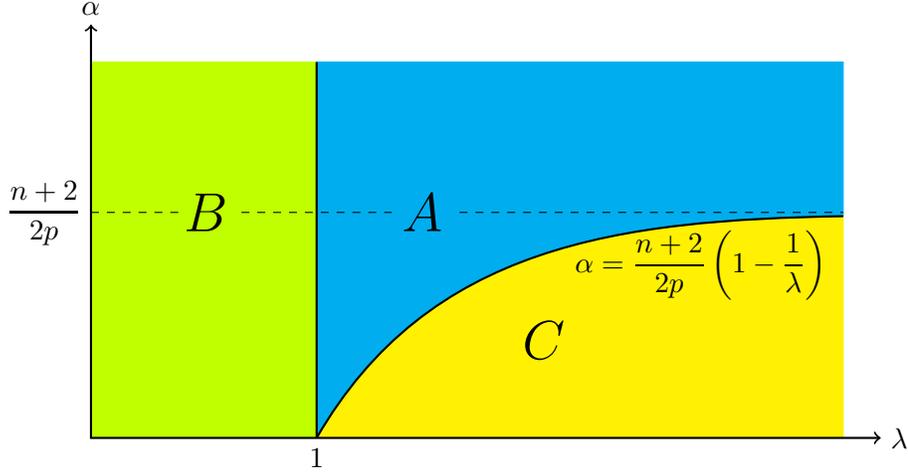
\begin{figure} 

\begin{tikzpicture}
\draw [fill=lime, lime] (0,0) rectangle (3,5);
\draw [fill=yellow, yellow] (3,0) to [out=60, in=181] (10,2.95) -- (10,0)
-- (3,0);
\draw [fill=cyan, cyan] (3,0) to [out=60, in=181] (10,2.95) -- (10,5)
-- (3,5) -- (3,0);

\draw [<->] [thick] (0,5.5) -- (0,0) -- (10.5,0);
\draw [thick] (3,0) -- (3,5);
\node [left] at (0,2.99) {$\hskip 0.6in \dfrac{n+2}{2p}$};
\node [below] at (3,0) {1};
\draw [dashed] (0,3) -- (1.2,3);
\draw [dashed] (2,3) -- (4,3);
\draw [dashed] (4.9,3) -- (10,3);

\draw [thick] (3,0) to [out=60, in=181] (10,2.95);
\node [right] at (6.3,2.3) {$\alpha=\dfrac{n+2}{2p}\left(1-\dfrac{1}{\lambda}\right)$};
\node [right] at (10.5,0) {$\lambda$};
\node [above] at (0,5.5) {$\alpha$};
\node [right] at (1.1,3) {\huge $B$};
\node [right] at (4.0,3) {\huge $A$};
\node [right] at (5.6,1.3) {\huge $C$};
\end{tikzpicture}

\caption{Graphs of the regions $A$, $B$, and $C$.}
\label{fig1}

\end{figure}

The following theorem deals with the case that $(\lambda,\alpha)\in A$.

\begin{thm}\label{thm3.1}
 Suppose $\alpha$ and $p$ satisfy \eqref{2.12}, $(\lambda, \alpha)\in A$, and $u$ satisfies \eqref{3.1}--\eqref{3.3}.  Then
 $$u=(\partial_t-\Delta)^\alpha u=0\quad\text{almost everywhere in }\mathbb{R}^n \times\mathbb{R}.$$
\end{thm}

The following three theorems deal with the case $(\lambda,\alpha)\in B$.

\begin{thm}\label{thm3.2}
 Suppose $\alpha$ and $p$ satisfy \eqref{2.12}, $(\lambda,\alpha)\in B$, and $u$ satisfies \eqref{3.1}--\eqref{3.3}.  Then for all $T>0$ we have
 \begin{equation}\label{3.4}
  \| u\|_{L^\infty (\mathbb{R}^n \times(0,T))}\leq(MT^\alpha )^{\frac{1}{1-\lambda}}
 \end{equation}
 and
 \begin{equation}\label{3.5}
  \|(\partial_t-\Delta)^\alpha u\|_{L^\infty (\mathbb{R}^n \times(0,T))}\leq(MT^\alpha )^{\frac{\lambda}{1-\lambda}}
 \end{equation}
 where
 \begin{equation}\label{3.6}
  M=M(\alpha,\lambda)=\frac{\Gamma(\frac{\alpha\lambda}{1-\lambda}+1)}{\Gamma(\alpha+\frac{\alpha\lambda}{1-\lambda}+1)}
 \end{equation}
where $\Gamma$ is the Gamma function.
\end{thm}

By the following theorem the bounds \eqref{3.4} and \eqref{3.5} in 
Theorem \ref{thm3.2} are optimal.

\begin{thm}\label{thm3.3}
 Suppose $\alpha$ and $p$ satisfy \eqref{2.12}, $(\lambda,\alpha)\in B,\,T>0$, and $N<M$ where $M$ is given by \eqref{3.6}.  Then there exists a solution
 $$u\in Y^p_\alpha \cap C(\mathbb{R}^n \times\mathbb{R})$$
 of \eqref{3.2}, \eqref{3.3} such that
 $$(\partial_t-\Delta)^\alpha u\in L^p (\mathbb{R}^n \times\mathbb{R})\cap C(\mathbb{R}^n \times\mathbb{R}),$$
 $$u(0,t)\geq(Nt^\alpha )^{\frac{1}{1-\lambda}}\quad\text{for }0<t<T$$
 and
 $$(\partial_t-\Delta)^\alpha u(0,t)=(Nt^\alpha )^{\frac{\lambda}{1-\lambda}}\quad\text{for }0<t<T.$$
\end{thm}

Although the estimates \eqref{3.4} and \eqref{3.5} are optimal there still
remains the question as to whether there is a {\it single} solution which
has the same size as these estimates as $t\to\infty$.  By the
following theorem there is such a solution.

\begin{thm}\label{thm3.4}
 Suppose $\alpha$ and $p$ satisfy \eqref{2.12} and $(\lambda,\alpha)\in B$.  Then there exists $N>0$ and $u\in Y^p_\alpha$ satisfying \eqref{3.2}, \eqref{3.3} such that
 $$u(x,t)\geq(Nt^\alpha )^{\frac{1}{1-\lambda}}\quad\text{for }(x,t)\in\Omega$$ 
 and
 $$(\partial_t-\Delta)^\alpha u(x,t)\geq(Nt^\alpha )^{\frac{\lambda}{1-\lambda}}\quad\text{for }(x,t)\in\Omega$$
 where $\Omega=\{(x,t)\in\mathbb{R}^n\times\mathbb{R} :|x|^2 <t\}$.
\end{thm}

According to the following theorem, if $(\lambda,\alpha)\in C$ then
there exist bounds as $t\to0^+$ for solutions of
\eqref{3.1}--\eqref{3.3} in neither the pointwise (i.e. $L^\infty$)
sense nor in the $L^q$ sense when $q>p$.

Moreover by Theorem \ref{thm3.6} the same is true as $t\to\infty$ provided $q\in[q_0 ,\infty]$ for some $q_0 =q_0(n,\alpha,\lambda)>p$.

\begin{thm}\label{thm3.5}
 Suppose $\alpha$ and $p$ satisfy \eqref{2.12} 
 $$(\lambda,\alpha)\in C \quad\text{ and }\quad  q\in(p,\infty].$$
 Then there exists a solution $u\in Y^p_\alpha$ of \eqref{3.2}, \eqref{3.3} and a sequence $\{t_j \}\subset(0,1)$ such that
 $$\lim_{j\to\infty}t_j =0$$
 and
 $$\|u^\lambda\|_{L^q(R_j )}=\|(\partial_t-\Delta)^\alpha u\|_{L^q (R_j )}=\infty\quad\text{for }j=1,2,...,$$
 where
 \begin{equation}\label{3.7}
  R_j =\{(x,t)\in\mathbb{R}^n \times\mathbb{R}:|x|<\sqrt{t_j} \text{ and }  t_j <t<2t_j \}.
 \end{equation}
\end{thm}

\begin{thm}\label{thm3.6}
 Suppose $\alpha$ and $p$ satisfy \eqref{2.12},
 $$(\lambda,\alpha)\in C \quad\text{ and }\quad  q\in\left[\frac{n+2}{2\alpha}\left(1-\frac{1}{\lambda}\right),\infty\right].$$
 Then there exists a solution $u\in Y^p_\alpha$ of \eqref{3.2}, \eqref{3.3} and a sequence $\{t_j \}\subset(1,\infty)$ such that
 $$\lim_{j\to\infty}t_j =\infty$$
 and
 $$\|u^\lambda\|_{L^q(R_j )}=\|(\partial_t-\Delta)^\alpha u\|_{L^q
   (R_j )}=\infty
\quad\text{for } j=1,2,...,$$
 where $R_j$ is given in \eqref{3.7}.
\end{thm}

\section{$J_\alpha$ version of fully fractional initial value problems}\label{sec4}
In order to prove our results stated in Section \ref{sec3}, we will
first reformulate them in terms of the inverse $J_\alpha$ of the fractional heat operator \eqref{2.14} as follows.

Suppose that $\lambda>0$ and, as assumed in Definition \ref{def2.1}
and Theorems \ref{thm3.1}--\ref{thm3.6}, that $p$ and $\alpha$ satisfy
\eqref{2.12}.  Then, by Theorem \ref{thm2.3}, $u$ satisfies
\eqref{3.1}--\eqref{3.3} if and only if
$f:=(\partial_t-\Delta)^\alpha u$ satisfies
\begin{equation}\label{4.1}
 f\in X^p
\end{equation}
\begin{equation}\label{4.2}
 0\leq f\leq(J_\alpha f)^\lambda \quad\text{in } \mathbb{R}^n \times\mathbb{R}
\end{equation}
\begin{equation}\label{4.3}
 f=0\quad\text{in } \mathbb{R}^n \times(-\infty,0).
\end{equation}

Thus the two problems \eqref{3.1}--\eqref{3.3} and
\eqref{4.1}--\eqref{4.3} are equivalent under the transformation
$u=J_\alpha f$ when $p$ and $\alpha$ satisfy \eqref{2.12}.  This
restriction on $p$ and $\alpha$ was imposed so that $J_\alpha f$ would
be defined pointwise in $\mathbb{R}^n \times\mathbb{R}$ for all
$f\in X^p$.  If $p\geq1$ and $\alpha>0$ do not satisfy \eqref{2.12},
that is, if 
\begin{equation}\label{4.3.5}
\left(p>1  \text{ and }   \alpha\geq\frac{n+2}{2p}\right)\quad\text{or} \quad \left(p=1
  \text{ and }  
\alpha>\frac{n+2}{2p}\right) 
\end{equation}
then $J_\alpha f$ is generally not defined pointwise as an extended
real valued function for 
$f\in X^p$.  (However
it can be defined for all $f$ in the subspace $L^p (\mathbb{R}^n \times\mathbb{R})$
of $X^p$
as a distribution on a certain subspace of the Schwarz space $S$
(see \cite[Sec 9.2.5]{SK}).

Even though $J_\alpha f$ is generally not defined pointwise as and extended
real valued function for $f\in X^p$
when $p$ and $\alpha$ satisfy \eqref{4.3.5}, it is defined pointwise
as a nonnegative extended real value function for all {\it
nonnegative} functions $f\in X^p$ for all $p\geq1$ and $\alpha>0$
because then the integrand of $J_\alpha f$ is a nonnegative function.
Hence, since $f$ is nonnegative in the problem
\eqref{4.1}--\eqref{4.3}, we see that the problem
\eqref{4.1}--\eqref{4.3} makes sense for all $p\geq 1$ and $\alpha>0$
when $J_\alpha$ is defined in the pointwise sense, which is the sense
in which we will define it in this section.  However
$J_\alpha$, when restricted to the set $X^{p}_{+}$ of all nonnegative
functions $f\in X^p$, is not one-to-one when $p$ and $\alpha$ satisfy
\eqref{4.3.5}.  Thus our results in this section for the problem
\eqref{4.1}--\eqref{4.3} when $p\geq1$ and $\alpha>0$ will yield
corresponding results for the problem \eqref{3.1}--\eqref{3.3} only
when $p$ and $\alpha$ satisfy \eqref{2.12}.

In view of these remarks, we will consider in this section solutions
\begin{equation}\label{4.4}
 f\in X^p
\end{equation}
of the following $J_\alpha$ version of the fully fractional initial
value problem \eqref{3.2}, \eqref{3.3}: 
\begin{equation}\label{4.5}
 0\leq f\leq K(J_\alpha f)^\lambda \quad\text{in } \mathbb{R}^n \times\mathbb{R},\,n\geq1
\end{equation}
\begin{equation}\label{4.6}
 f=0\quad\text{in } \mathbb{R}^n \times(-\infty, 0)
\end{equation}
where
\begin{equation}\label{4.7}
 p\in[1,\infty) \quad\text{ and }\quad   K,\lambda,\alpha\in(0,\infty)
\end{equation}
are constants, $X^p$ is defined by \eqref{2.9}, and $J_\alpha$ is given by \eqref{2.4}.

Under the equivalence of problems \eqref{3.1}--\eqref{3.3} and \eqref{4.1}--\eqref{4.3} discussed above, the following Theorems \ref{thm4.1}--\ref{thm4.6}, when restricted to the case that $p$ and $\alpha$ satisfy \eqref{2.12} and $ K=1$, clearly imply Theorems \ref{thm3.1}--\ref{thm3.6} in 
Section \ref{sec3}.  We will prove Theorems \ref{thm4.1}--\ref{thm4.6} in Section \ref{sec8}.

\begin{thm}\label{thm4.1}
 Suppose $(\lambda,\alpha)\in A$ and $f,p$, and $ K$ satisfy \eqref{4.4}--\eqref{4.7}.  Then
 \begin{equation}\label{4.8}
  f=J_\alpha f=0\quad\text{almost everywhere in }\mathbb{R}^n \times\mathbb{R}.
 \end{equation}
\end{thm}

\begin{thm}\label{thm4.2}
 Suppose $(\lambda,\alpha)\in B$ and $f,p$, and $ K$ satisfy \eqref{4.4}--\eqref{4.7}.  Then for all $b>0$ we have
 \begin{equation}\label{4.9}
  \| f\|_{L^\infty (\mathbb{R}^n \times(0,b))}\leq K^{\frac{1}{1-\lambda}}(Mb^\alpha )^{\frac{\lambda}{1-\lambda}}
 \end{equation}
 and
 \begin{equation}\label{4.10}
  \| J_\alpha f\|_{L^\infty (\mathbb{R}^n \times(0,b))}\leq K^{\frac{1}{1-\lambda}}(Mb^\alpha )^{\frac{1}{1-\lambda}}
 \end{equation}
 where
 \begin{equation}\label{4.11}
  M=M(\alpha,\lambda)=\frac{\Gamma(\frac{\alpha\lambda}{1-\lambda}+1)}{\Gamma(\alpha+\frac{\alpha\lambda}{1-\lambda}+1)}.
 \end{equation}
\end{thm}

\begin{thm}\label{thm4.3}
 Suppose $p$ and $ K$ satisfy \eqref{4.7}, $(\lambda,\alpha)\in B$, $T>0$, and $0<N<M$ where $M$ is given by \eqref{4.11}.  Then there exists a solution
 \begin{equation}\label{4.12}
  f\in L^p (\mathbb{R}^n \times\mathbb{R})\cap C(\mathbb{R}^n \times\mathbb{R})
 \end{equation}
 of \eqref{4.5}, \eqref{4.6} such that
 \begin{equation}\label{4.12.5}
  J_\alpha f\in C(\mathbb{R}^n \times\mathbb{R})
 \end{equation}
 \begin{equation}\label{4.13}
  f(0,t)= K^{\frac{1}{1-\lambda}}(Nt^\alpha )^{\frac{\lambda}{1-\lambda}}\quad\text{for }0<t<T
 \end{equation}
 and
 \begin{equation}\label{4.14}
  J_\alpha f(0,t)\geq K^{\frac{1}{1-\lambda}}(Nt^\alpha )^{\frac{1}{1-\lambda}}\quad\text{for }0<t<T.
 \end{equation}
\end{thm}

\begin{thm}\label{thm4.4}
 Suppose $p$ and $ K$ satisfy \eqref{4.7} and $(\lambda,\alpha)\in B$.  Then there exists $N>0$ and
 $$f\in X^p$$
 satisfying \eqref{4.5}, \eqref{4.6} such that
 \begin{equation}\label{4.15}
  f(x,t)\geq K^{\frac{1}{1-\lambda}}(Nt^\alpha )^{\frac{\lambda}{1-\lambda}}\quad\text{for }|x|^2 <t
 \end{equation}
 and
 \begin{equation}\label{4.16}
  J_\alpha f(x,t)\geq K^{\frac{1}{1-\lambda}}(Nt^\alpha )^{\frac{1}{1-\lambda}}\quad\text{for }|x|^2 <t.
 \end{equation}
\end{thm}

\begin{thm}\label{thm4.5}
 Suppose $p$ and $ K$ satisfy \eqref{4.7},
\begin{equation}\label{4.17}
(\lambda,\alpha)\in C\qquad and \qquad q\in(p,\infty].
\end{equation} 
Then there exists a solution
\begin{equation}\label{4.18}
  f\in L^p (\mathbb{R}^n \times\mathbb{R})
 \end{equation}
 of \eqref{4.5}, \eqref{4.6} and a sequence $\{t_j \}\subset(0,1)$ such that
 $$\lim_{j\to\infty}t_j =0$$ 
 and
 \begin{equation}\label{4.19}
  \| f\|_{L^q (R_j )}=\infty\quad\text{for }j=1,2,...,
 \end{equation}
 where
\begin{equation}\label{4.19.5} 
R_j =\{(x,t)\in\mathbb{R}^n
  \times\mathbb{R}:|x|<\sqrt{t_j} \text{ and }  t_j <t<2t_j \}.
\end{equation}
\end{thm}

\begin{thm}\label{thm4.6}
 Suppose $p$ and $ K$ satisfy \eqref{4.7},
 \begin{equation}\label{4.20}
  (\lambda,\alpha)\in C \quad\text{ and }\quad  
\frac{n+2}{2\alpha}(1-\frac{1}{\lambda})\le q\le\infty.
 \end{equation}
 Then there exists a solution
 \begin{equation}\label{4.21}
  f\in X^p
 \end{equation}
 of \eqref{4.5}, \eqref{4.6} and a sequence $\{t_j \}\subset(1,\infty)$ such that
 $$\lim_{j\to\infty}t_j =\infty$$
 and
 \begin{equation}\label{4.22}
  \| f\|_{L^q (R_j )}=\infty\quad\text{for }j=1,2,...,
 \end{equation}
 where $R_j$ is given in \eqref{4.19.5}.
 
\end{thm}

\section{Preliminary results for fully fractional heat operators}\label{sec5}
In this section we provide some lemmas needed for the proofs of our
results in Section \ref{sec2} concerning the fully fractional heat
operator \eqref{2.14}.

The following lemma is needed for the proof of Theorem \ref{thm2.2}.
\begin{lem}\label{lem5.1}
 Suppose $\alpha,\beta>0$.  Then 
 \begin{equation}\label{5.1}
  \Phi_{\alpha+\beta}=\Phi_\alpha *\Phi_\beta 
\quad\text{in } \mathbb{R}^n \times\mathbb{R}
 \end{equation}
 where $\Phi_\alpha$ is defined in \eqref{2.2}.
\end{lem}

\begin{proof}
 Since
 \begin{align}\label{5.2}
  \notag &\Phi_\alpha *\Phi_\beta (x,t)=\int^{\infty}_{-\infty}\int_{\xi\in\mathbb{R}^n}\Phi_\alpha (x-\xi,t-\tau)\Phi_\beta (\xi,\tau)\, d\xi \, d\tau\\
  &=
  \begin{cases}
   0 & \text{for }(x,t)\in\mathbb{R}^n \times(-\infty,0]\\
   \int^{t}_{0}\int_{\xi\in\mathbb{R}^n}\Phi_\alpha (x-\xi,t-\tau)\Phi_\beta (\xi,\tau)\, d\xi \, d\tau & \text{for }(x,t)\in\mathbb{R}^n \times(0,\infty),
  \end{cases}
 \end{align}
 we have \eqref{5.1} holds in $\mathbb{R}^n \times(-\infty,0]$.
 
 Using the well-known facts that
\begin{equation}\label{5.2.5} 
\widehat{\Phi}_\alpha (\cdot,t)(y)=\frac{t^{\alpha-1}}{\Gamma(\alpha)}e^{-t|y|^2}\quad\text{for }t>0 \text{ and }  y\in\mathbb{R}^n
 \end{equation}
 and
 \begin{equation}\label{5.3}
  \int^{t}_{0}\frac{(t-\tau)^{\alpha-1}\tau^{\beta-1}}{\Gamma(\alpha)\Gamma(\beta)}d\tau =\frac{t^{\alpha+\beta-1}}{\Gamma(\alpha+\beta)}\quad\text{for }t,\alpha,\beta>0,
 \end{equation}
 and assuming we can interchange the order of integration in the following calculation (we will justify this after the calculation) we obtain for $t>0$ and $y\in\mathbb{R}^n$ that
 \begin{align}
  \notag &(\Phi_\alpha *\Phi_\beta )\widehat{\phantom{\Phi}}(\cdot,t)(y)\\
 \label{5.4}
 &=\int_{x\in\mathbb{R}^n}e^{ix\cdot y}\int^{t}_{0}\Biggl(\int_{\xi\in\mathbb{R}^n}\Phi_\alpha
         (x-\xi,t-\tau)\Phi_\beta (\xi,\tau)\, d\xi\Biggr)\, d\tau\, dx\\
  \notag &=\int^{t}_{0}\Biggl(\int_{x\in\mathbb{R}^n}e^{ix\cdot y}\Biggl(\int_{\xi\in\mathbb{R}^n}\Phi_\alpha (x-\xi,t-\tau)\Phi_\beta (\xi,\tau)\,d\xi\Biggr)dx\Biggr)d\tau\\
 \notag &=\int^{t}_{0}\frac{(t-\tau)^{\alpha-1}}{\Gamma(\alpha)}e^{-|y|^2
           (t-\tau)}\frac{\tau^{\beta-1}}{\Gamma(\beta)}e^{-|y|^2
           \tau}d\tau\quad\text{(by the convolution theorem)}\\
\notag &=e^{-|y|^2 t}\int^{t}_{0}\frac{(t-\tau)^{\alpha-1}\tau^{\beta-1}}{\Gamma(\alpha)\Gamma(\beta)}d\tau\\
  &=e^{-t|y|^2}\frac{t^{\alpha+\beta-1}}{\Gamma(\alpha+\beta)}=\widehat{\Phi}_{\alpha+\beta}(\cdot,t)(y)\label{5.5}.
 \end{align}
 This calculation is justified by Fubini's theorem and the fact that
 the integral \eqref{5.4} with $e^{ix\cdot y}$ replaced with $1$ is, by Fubini's theorem for nonnegative functions and \eqref{5.3}, equal to
 \begin{align*}
  &\int^{t}_{0}\int_{\xi\in\mathbb{R}^n}\biggl(\int_{x\in\mathbb{R}^n}\Phi_\alpha (x-\xi,t-\tau)dx\Biggr)\Phi_\beta (\xi,\tau)\, d\xi \, d\tau\\
  &=\int^{t}_{0}\int_{\xi\in\mathbb{R}^n}\frac{(t-\tau)^{\alpha-1}}{\Gamma(\alpha)}\Phi_\beta (\xi,\tau)\, d\xi \, d\tau\\
  &=\frac{t^{\alpha+\beta-1}}{\Gamma(\alpha+\beta)}\quad\text{for }t>0 \text{ and }  y\in\mathbb{R}^n .
 \end{align*}
 It follows now from \eqref{5.5} that \eqref{5.1} holds in $\mathbb{R}^n \times(0,\infty)$.
\end{proof}

The following lemma is needed for the proof of Lemma \ref{lem5.3}
which in turn is needed for the proof of Theorem \ref{thm2.3}.
\begin{lem}\label{lem5.2}
 Suppose $f\in L^1 (-\infty,0)$ and $0<\alpha\leq1$.  Then
 $$g(t):=\int^{t}_{-\infty}(t-\tau)^{\alpha-1}|f(\tau)|d\tau<\infty\quad\text{ for almost all }t\in(-\infty,0).$$
\end{lem}

\begin{proof}
 The lemma is clearly true if $\alpha=1$.  Hence we can assume $0<\alpha<1$.  Since
 \begin{align*}
  \int^{0}_{-\infty}&(-t)^{-\alpha}g(t)\,dt
=\int^{0}_{-\infty}(-t)^{-\alpha}\int^{t}_{-\infty}(t-\tau)^{\alpha-1}|f(\tau)|\, d\tau \, dt\\
  &=\int^{0}_{-\infty}|f(\tau)|\biggl(\int^{0}_{\tau}(-t)^{(1-\alpha)-1}(t-\tau)^{\alpha-1}dt\Biggr)d\tau\\
  &=\Gamma(1-\alpha)\Gamma(\alpha)\int^{0}_{-\infty}|f(\tau)|\,d\tau<\infty,
 \end{align*}
 where we have used \eqref{5.3}, we see that
 $g(t)<\infty$ for almost all $t\in(-\infty,0)$.
\end{proof}

\begin{lem}\label{lem5.3}
 Suppose $f\in L^1 (\mathbb{R}^n \times(-\infty,0)),\,\alpha\in(0,1]$, and $y\in\mathbb{R}^n$.  Then for almost all $t\in(-\infty,0)$ we have
\[
\widehat{J_\alpha f}(\cdot,t)(y)
=\int^{t}_{-\infty}\frac{(t-\tau)^{\alpha-1}}{\Gamma(\alpha)}
e^{-|y|^2 (t-\tau)}   \widehat{f}(\cdot,\tau)(y)\,d\tau.
\]
\end{lem}

\begin{proof}
 By Fubini's theorem for nonnegative functions and Lemma \ref{lem5.2} we find for almost all $t\in(-\infty,0)$ that
 \begin{align*}
  \int_{x\in\mathbb{R}^n}&|e^{ix\cdot
                           y}|\int^{t}_{-\infty}\frac{(t-\tau)^{\alpha-1}}{\Gamma(\alpha)}\int_{\xi\in\mathbb{R}^n}\Phi_1(x-\xi,t-\tau)|f(\xi,\tau)| \,d\xi\, d\tau \,dx\\
  &=\int^{t}_{-\infty}\frac{(t-\tau)^{\alpha-1}}{\Gamma(\alpha)}\Biggl(\int_{\xi\in\mathbb{R}^n}|f(\xi,\tau)|d\xi\Biggr)d\tau<\infty.
 \end{align*}
 Hence by Fubini's theorem, the convolution theorem for Fourier
 transforms, and \eqref{5.2.5}, we see for almost all $t\in(-\infty,0)$ that
 \begin{align*}
\widehat{J_\alpha f}(\cdot,t)(y) 
  &=\int^{t}_{-\infty}\int_{x\in\mathbb{R}^n}e^{ix\cdot y}\int_{\xi\in\mathbb{R}^n}\Phi_\alpha(x-\xi,t-\tau)f(\xi,\tau)\,d\xi\, dx\,d\tau\\
  &=\int^{t}_{-\infty}\frac{(t-\tau)^{\alpha-1}}{\Gamma(\alpha)}e^{-|y|^2
    (t-\tau)}
 \widehat{f}(\cdot,\tau)(y)\,d\tau.
 \end{align*}
\end{proof}

\section{Fully fractional heat operator proofs}\label{sec6}
In this section we prove our fully fractional heat operator results which we stated in Section \ref{sec2}.

\begin{proof}[Proof of Theorem \ref{thm2.1}] Part (i) was proved by 
Sampson \cite[Theorem 2.2]{SP}.  We prove part (ii) in two steps.
\medskip

\noindent\underline{Step 1.}  Suppose $f,g\in S$.  Let $(x,t)\in\mathbb{R}^n \times\mathbb{R}$ be momentarily fixed and define $\varphi\in S$ by
$$\varphi(y,s)=f(x+y,t+s).$$
Then
$$\widehat{\widehat{\varphi}}(y,s)=(2\pi)^{n+1}\varphi(-y,-s)=(2\pi)^{n+1}f(x-y,t-s)$$
and
$$\widehat{\varphi}(y,s)=e^{-ix\cdot y-its}\widehat{f}(y,s).$$
Thus by part (i) with $\varphi$ replaced with $\widehat{\varphi}$ we get
\begin{align}\label{6.1}
 \notag (2\pi)^{n+1}J_\alpha f(x,t)&=(2\pi)^{n+1}\iint_{\mathbb{R}^n \times\mathbb{R}}\Phi_\alpha (y,s)f(x-y,t-s)\, dy \, ds\\
 \notag &=\iint_{\mathbb{R}^n \times\mathbb{R}}\Phi_\alpha (y,s)\widehat{\widehat{\varphi}}(y,s)\, dy \, ds\\
 \notag &=\iint_{\mathbb{R}^n \times\mathbb{R}}(|y|^2 -is)^{-\alpha}\widehat{\varphi}(y,s)\, dy \, ds\\
 &=\iint_{\mathbb{R}^n \times\mathbb{R}}(|y|^2 -is)^{-\alpha}\widehat{f}(y,s)e^{-ix\cdot y-its}\, dy \, ds.
\end{align}
Multiplying \eqref{6.1} by $\widehat{g}(x,t)/(2\pi)^{n+1}$,
integrating the resulting equation with respect to $(x,t)$, and
interchanging the order of integration in the resulting integral on
the RHS, which is allowed by Fubini's theorem and the fact that 
\begin{equation}\label{6.1.5}
\iint_{||y|^2-is|\le 1}||y|^2-is|^{-\alpha}dy\,ds<\infty\quad\text{for }
0<\alpha<(n+2)/2,
\end{equation}
we get \eqref{2.5}.
\medskip

\noindent\underline{Step 2.} Suppose $f\in L^1 (\mathbb{R}^n
\times\mathbb{R})$ and $g\in S$.  Then $\widehat{g}\in S$ and 
$\widehat{f}\in C(\mathbb{R}^n \times\mathbb{R})
\cap L^\infty (\mathbb{R}^n\times\mathbb{R} )$.  Since $S$ is dense in 
$L^1 (\mathbb{R}^n \times\mathbb{R})$ there exists $\{f_j \}\subset S$ such that $f_j \to f$ in $L^1 (\mathbb{R}^n \times\mathbb{R})$ and by Step 1
\begin{equation}\label{6.2}
 \iint_{\mathbb{R}^n \times\mathbb{R}}J_\alpha f_j (x,t)\widehat{g}(x,t)\, dx \, dt=\iint_{\mathbb{R}^n \times\mathbb{R}}(|y|^2 -is)^{-\alpha}\widehat{f}_j (y,s)g(y,s)\, dy \, ds.
\end{equation}
Since
$$\|\widehat{f}_j -\widehat{f}\|_{L^\infty (\mathbb{R}^n
  \times\mathbb{R})}\leq\| f_j -f\|_{L^1 (\mathbb{R}^n
  \times\mathbb{R})}\to 0\quad\text{as }j\to\infty$$
we have
\begin{align*}
 \Biggl|&\iint_{\mathbb{R}^n \times\mathbb{R}}(\widehat{f}_j (y,s)-\widehat{f}(y,s))(|y|^2 -is)^{-\alpha}g(y,s)\, dy \, ds\Biggr|\\
 &\leq\|\widehat{f}_j -\widehat{f}\|_{L^\infty (\mathbb{R}^n \times\mathbb{R})}\iint_{\mathbb{R}^n \times\mathbb{R}}||y|^2 -is|^{-\alpha}|g(y,s)|\, dy \, ds\\
&\to0\quad\text{as }j\to\infty
\end{align*}
by \eqref{6.1.5}.  Thus the RHS of \eqref{6.2} tends to the RHS of \eqref{2.5} as $j\to\infty$.

Also, defining $h(x,t)=|\widehat{g}(-x,-t)|$ we have 
\begin{align*}
 \Big|\iint_{\mathbb{R}^n \times\mathbb{R}}J_\alpha (f_j
  &-f)(x,t)\widehat{g}(x,t)
\, dx \, dt\Big|\\
 &\leq\iint_{\mathbb{R}^n \times\mathbb{R}}\iint_{\mathbb{R}^n
   \times\mathbb{R}}
\Phi_\alpha (x-y,t-s)|(f_j -f)(y,s)|\, dy \, ds\,|\widehat{g}(x,t)|\, dx \, dt\\
 &=\iint_{\mathbb{R}^n \times\mathbb{R}}|(f_j -f)(y,s)|
(\Phi_\alpha*h)(-y,-s)
\,dy\,ds\\
&\to 0\quad\text{as }j\to\infty
\end{align*}
because noting that
$h\in L^1(\mathbb{R}^n\times\mathbb{R})\cap L^\infty(\mathbb{R}^n\times\mathbb{R})$,
\begin{align}
 \notag \|\Phi_\alpha \raisebox{2pt}{$\chi$}_{\mathbb{R}^n \times(0,1)}\|_{L^1 (\mathbb{R}^n \times\mathbb{R})}&=\int^1_0\frac{t^{\alpha-1}}{\Gamma(\alpha)}\int_{x\in\mathbb{R}^n}\Phi_1(x,t)\, dx \, dt\\
 \label{6.2.5} &=\int^1_0\frac{t^{\alpha-1}}{\Gamma(\alpha)}dt<\infty
\quad\text{for }\alpha>0,
\end{align}
and
$\Phi_\alpha\raisebox{2pt}{$\chi$}_{\mathbb{R}^n\times(1,\infty)}\in
L^\infty(\mathbb{R}^n\times\mathbb{R})$ for $\alpha<(n+2)/2$
we find that
\[
\Phi_\alpha *h 
=\Phi_\alpha\raisebox{2pt}{$\chi$}_{\mathbb{R}^n\times(0,1)}*h
+\Phi_\alpha\raisebox{2pt}{$\chi$}_{\mathbb{R}^n\times(1,\infty)}*h\in L^\infty(\mathbb{R}^n\times\mathbb{R})
\]
by Young's inequality.
Thus the LHS of \eqref{6.2} tends to the LHS of
\eqref{2.5} as $j\to\infty$. 
\end{proof}

\begin{proof}[Proof of Theorem \ref{thm2.2}] 
 Since
\[
\bigcap_{T\in\mathbb{R}}L^{p}_{\text{loc}}(\mathbb{R}^n \times \mathbb{R}_T )
=L^{p}_{\text{loc}}(\mathbb{R}^n \times\mathbb{R})
\]
and since 
$(J_\alpha f_T )|_{\mathbb{R}^n \times\mathbb{R}_T}
=(J_\alpha f)|_{\mathbb{R}^n \times\mathbb{R}_T}$, 
where 
$f_T =f\raisebox{2pt}{$\chi$}_{\mathbb{R}^n \times \mathbb{R}_T}$ to prove (i), (ii) and (iii) it 
suffices to prove for all $T\in\mathbb{R}$ that
\begin{enumerate}
 \item[(i)$^\prime$] $J_\alpha f_T ,\,J_\alpha |f_T |\in L^{p}_{\text{loc}}(\mathbb{R}^n \times \mathbb{R}_T )$
 \item[(ii)$^\prime$] $J_\beta J_\gamma f_T =J_\alpha f_T \quad\text{in } L^{p}_{\text{loc}}(\mathbb{R}^n \times\mathbb{R}_T )$ whenever $\beta>0,\,\gamma>0$, and $\beta+\gamma=\alpha$\\
 and
 \item[(iii)$^\prime$] $HJ_\alpha f_T =J_{\alpha-1}f_T \quad\text{in } D^\prime (\mathbb{R}^n \times\mathbb{R}_T)$ when $\alpha>1$.
\end{enumerate}
To do this, let $T\in\mathbb{R}$ be fixed.  Since $f\in X^p \subset L^p (\mathbb{R}^n \times \mathbb{R}_T )$ we have
\begin{equation}\label{6.3}
 f_T \in L^p (\mathbb{R}^n \times\mathbb{R}).
\end{equation}
\medskip

\noindent\underline{Proof of (i)$^\prime$.} Since $|J_\alpha f_T |\leq J_\alpha |f_T |$, to prove (i)$^\prime$ it suffices to prove only that
\begin{equation}\label{6.4}
 J_\alpha |f_T |\in L^{p}_{\text{loc}}(\mathbb{R}^n \times\mathbb{R}_T ).
\end{equation}
By \eqref{2.3} we have
\begin{equation}\label{6.5}
 J_\alpha |f_T |=u_1 +u_2 ,
\end{equation}
where
$$u_1 =(\Phi_\alpha \raisebox{2pt}{$\chi$}_{\mathbb{R}^n \times(0,1)})*|f_T| \quad\text{and}\quad  u_2 =(\Phi_\alpha \raisebox{2pt}{$\chi$}_{\mathbb{R}^n \times(1,\infty)})*|f_T|.$$
It follows from \eqref{6.2.5},  \eqref{6.3}, and Young's inequality that
\[
u_1 \in L^p (\mathbb{R}^n \times\mathbb{R}).
\]
Thus to complete the proof of \eqref{6.4} and hence of (i)$^\prime$ it suffices to show 
\begin{equation}\label{6.6}
 u_2 \in L^\infty (\mathbb{R}^n \times\mathbb{R}).
\end{equation}
To do this we consider two cases.
\medskip

\noindent\underline{Case I.} Suppose $1<p<\frac{n+2}{2\alpha}$.  Let $q$ be the conjugate H\"older exponent for $p$.  Then
\[\frac{1}{q}=1-\frac{1}{p}<1-\frac{2\alpha}{n+2}=\frac{n+2-2\alpha}{n+2}\]
and thus making the change of variables $\sqrt{\frac{q}{4s}}y=z$   we obtain
\begin{align*}
 \|\Phi_\alpha \raisebox{2pt}{$\chi$}_{\mathbb{R}^n \times(1,\infty)}\|&^{q}_{L^q
                 (\mathbb{R}^n \times\mathbb{R})}
=C(n,\alpha,q)\int^{\infty}_{1}\int_{y\in\mathbb{R}^n}s^{(\alpha-1-n/2)q}e^{-\frac{q}{4s}|y|^2}\, dy \, ds\\
 &=C(n,\alpha,q)\int^{\infty}_{1}s^{(\alpha-1-n/2)q+n/2}\int_{z\in\mathbb{R}^n}e^{-|z|^2}dz\,ds
 <\infty.
\end{align*}
Hence \eqref{6.6} follows from \eqref{6.3} and Young's inequality.\\
\medskip

\noindent\underline{Case II.} Suppose $1=p\leq\frac{n+2}{2\alpha}$.  Then
\begin{align*}
 \Phi_\alpha \raisebox{2pt}{$\chi$}_{\mathbb{R}^n \times(1,\infty)}(y,s)&\leq C(n,\alpha)s^{\alpha-1-n/2}\raisebox{2pt}{$\chi$}_{\mathbb{R}^n \times(1,\infty)}(y,s)\\
 &\leq C(n,\alpha)\quad\text{for }(y,s)\in\mathbb{R}^n \times\mathbb{R}.
\end{align*}
Thus \eqref{6.6} follows from \eqref{6.3} and so the proof of (i)$^\prime$ is complete.\\
\medskip

\noindent\underline{Proof of (ii)$^\prime$.} Using Fubini's theorem for nonnegative functions and Lemma \ref{lem5.1} we have
\begin{align*}
 J_\beta (J_\gamma |f_T |)(x,t)&=\iint_{\mathbb{R}^n \times\mathbb{R}}\Phi_\beta (x-\xi,t-\tau)\iint_{\mathbb{R}^n \times\mathbb{R}}\Phi_\gamma (\xi-\eta,\tau-\zeta)|f_T (\eta,\zeta)|\, d\eta \, d\zeta \, d\xi \, d\tau\\
 &=\iint_{\mathbb{R}^n \times\mathbb{R}}\Phi_{\beta+\gamma}(x-\eta,t-\zeta)|f_T (\eta,\zeta)|\, d\eta \, d\zeta\\
 &=(J_\alpha |f_T |)(x,t)<\infty\quad\text{a.e. in }\mathbb{R}^n \times\mathbb{R}
\end{align*}
by part (i)$^\prime$.  Hence by Fubini's theorem the above calculation can be repeated with $|f_T|$ replaced with $f_T$ which gives (ii)$^\prime$.
\medskip

\noindent\underline{Proof of (iii)$^\prime$.} By (i)$^\prime$ we have
\begin{equation}\label{6.7}
 J_\alpha |f_T| ,\,J_{\alpha-1}|f_T |\in L^{p}_{\text{loc}}(\mathbb{R}^n \times\mathbb{R}_T )\subset D^\prime (\mathbb{R}^n \times\mathbb{R}_T ).
\end{equation}
Let $\varphi\in C^{\infty}_{0}(\mathbb{R}^n \times\mathbb{R}_T )$.  Then noting that
\begin{equation}\label{6.8}
 \iint_{\mathbb{R}^n \times\mathbb{R}_T}\Phi_1 (x-\eta,t-\zeta)H^* \varphi(x,t)\, dx \, dt=\varphi(\eta,\zeta)\quad\text{for }(\eta,\zeta)\in\mathbb{R}^n \times\mathbb{R}_T
\end{equation}
where $H^* =-\partial_t-\Delta$ and assuming we can interchange the order of integration in the following calculation (we will justify this after the calculation) it follows from Lemma \ref{lem5.1} that
\begin{align}
\label{6.9}&(H(J_\alpha f_T ))(\varphi)=(J_\alpha f_T )(H^* \varphi)\\
\notag &=\iint_{\mathbb{R}^n \times\mathbb{R}_T}\Biggl(\iint_{\mathbb{R}^n \times\mathbb{R}_T}\Phi_\alpha (x-\xi,t-\tau)f_T (\xi,\tau)\, d\xi \, d\tau\Biggr)H^* \varphi(x,t)\, dx \, dt\\
\notag&=\iint_{\mathbb{R}^n \times\mathbb{R}_T}\iint_{\mathbb{R}^n \times\mathbb{R}_T}\Biggl(\iint_{\mathbb{R}^n \times\mathbb{R}_T}\Phi_1 (x-\eta,t-\zeta)\Phi_{\alpha-1}(\eta-\xi,\zeta-\tau)\, d\eta \, d\zeta\Biggr)\\
\notag&\phantom{=\iint_{\mathbb{R}^n \times\mathbb{R}_T}}\times f_T (\xi,\tau)\, d\xi \, d\tau H^* \varphi(x,t)\, dx \, dt\\
\notag&=\iint_{\mathbb{R}^n \times\mathbb{R}_T}\Biggl(\iint_{\mathbb{R}^n \times\mathbb{R}_T}\Biggl(\iint_{\mathbb{R}^n \times\mathbb{R}_T}\Phi_1 (x-\eta,t-\zeta)H^* \varphi(x,t)\, dx \, dt\Biggr)\Phi_{\alpha-1}(\eta-\xi,\zeta-\tau)\, d\eta \, d\zeta\Biggr)\\
\label{6.10}&\phantom{=\iint_{\mathbb{R}^n \times\mathbb{R}_T}}\times f_T (\xi,\tau)\, d\xi \, d\tau \\
\notag&=\iint_{\mathbb{R}^n \times\mathbb{R}_T}\Biggl(\iint_{\mathbb{R}^n \times\mathbb{R}_T}\Phi_{\alpha-1}(\eta-\xi,\zeta-\tau)f_T (\xi,\tau)\, d\xi \, d\tau\Biggr)\varphi(\eta,\zeta)\, d\eta \, d\zeta\\
\notag&=(J_{\alpha-1}f_T )(\varphi).
\end{align}

To justify this calculation, it suffices by Fubini's theorem to show the integral \eqref{6.10}, with $f_T$ and $H^* \varphi$ replaced with $|f_T |$ and $|H^* \varphi|$, is finite.  However in the same way that \eqref{6.10} was obtained from \eqref{6.9}, we 
see that this modified integral equals
$$\iint_{\mathbb{R}^n \times\mathbb{R}_T}(J_\alpha |f_T|)(x,t)|H^* \varphi|(x,t)\, dx \, dt<\infty$$
by \eqref{6.7}.
\end{proof}

\begin{proof}[Proof of Theorem \ref{thm2.3}]
  Clearly (ii) implies (i).  We now prove (ii).  Suppose \eqref{2.13}.  It follows from \eqref{2.4} that
$$f|_{\mathbb{R}^n \times\mathbb{R}_T}=0\text{ implies }(J_\alpha f)|_{\mathbb{R}^n \times\mathbb{R}_T}=0.$$
Conversely suppose
\begin{equation}\label{6.11}
 (J_\alpha f)|_{\mathbb{R}^n \times\mathbb{R}_T}=0.
\end{equation}
The complete the proof of (ii) it suffices to prove
\begin{equation}\label{6.12}
 f|_{\mathbb{R}^n \times\mathbb{R}_T}=0.
\end{equation}
By Theorem \ref{thm2.2}(iii) and mathematical induction, we can, without loss of generality, assume for the proof \eqref{6.12} that
\begin{equation}\label{6.13}
 0<\alpha\leq1.
\end{equation}
Moreover, by translating we can assume
\begin{equation}\label{6.14}
 T=0.
\end{equation}
We divide the proof of \eqref{6.12} into two cases.\\
\medskip

\noindent\underline{Case I.} Suppose \eqref{2.12}$_2$ holds. Then
\begin{equation}\label{6.15}
 1=p\leq\frac{n+1}{2\alpha}.
\end{equation}
Let
\begin{equation}\label{6.16}
 F(y,t)=\widehat{f}(\cdot,t)(y)\quad
\text{for }(y,t)\in\mathbb{R}^n \times(-\infty,0).
\end{equation}
By \eqref{2.13} and \eqref{6.15} we have
\begin{equation}\label{6.17}
 f\in L^1 (\mathbb{R}^n \times(-\infty,0))
\end{equation}
and thus
\[f(\cdot,t)\in L^1 (\mathbb{R}^n )\quad\text{for almost all }t\in(-\infty,0)\]
which implies
\[F(\cdot,t)\in C(\mathbb{R}^n )\quad\text{for almost all }t\in(-\infty,0).\]
Also, by \eqref{6.17}
\begin{align}\label{6.18}
 \notag\| F(y,\cdot)\|_{L^1 (-\infty,0)}&=\int^{0}_{-\infty}\left|\int_{\mathbb{R}^n}e^{ix\cdot y}f(x,t)\,dx\right|\,dt\\
 &\leq\| f\|_{L^1 (\mathbb{R}^n \times(-\infty,0)}<\infty\quad\text{ for all }y\in\mathbb{R}^n .
\end{align}
\medskip

\noindent\underline{Case I(a).} Suppose $\alpha=1$.  Then by
\eqref{6.17}, \eqref{6.11}, and Lemma \ref{lem5.3}  we have
for each $y\in\mathbb{R}^n$ that
$$\int^{t}_{-\infty}e^{|y|^2 \tau}F(y,\tau)\,d\tau=e^{|y|^2 t}\int^{t}_{-\infty}e^{-|y|^2 (t-\tau)}F(y,\tau)\,d\tau=0$$
for almost all $t\in(-\infty,0)$.  Hence, by \eqref{6.18} and the measure theoretic fundamental theorem of calculus, we get $F=0$ in $L^1 (\mathbb{R}^n \times(-\infty,0))$ which together with \eqref{6.16} implies \eqref{6.12}.
\medskip

\noindent\underline{Case I(b).} Suppose $0<\alpha<1$.  To handle this case we hold $y\in\mathbb{R}^n \backslash\{0\}$ fixed and define
\begin{equation}\label{6.19}
 F_0 (t):=F(y,t).
\end{equation}
Then by \eqref{6.18}
\begin{equation}\label{6.20}
 F_0 \in L^1 (-\infty,0).
\end{equation}
From \eqref{6.17}, \eqref{6.11}, and Lemma \ref{lem5.3} we have
\begin{equation}\label{6.21}
 g(t):=\int^{t}_{-\infty}(t-\tau)^{\alpha-1}e^{|y|^2 \tau}F_0 (\tau)d\tau=0
\end{equation}
for almost all $t\in(-\infty,0)$.  On the other hand, assuming we can interchange the order of integration in the following calculation (we will justify this after the calculation), we find for $b\in\mathbb{R}$ that
\begin{align}\label{6.22}
 \notag &\int^{0}_{-\infty}\Biggl(\int^{0}_{t}(\zeta-t)^{-\alpha}\cos b\zeta \,d\zeta\Biggr)g(t)\,dt\\
\notag &=\int^{0}_{-\infty}e^{|y|^2 \tau}F_0 (\tau)\Biggl(\int^{0}_{\tau}\cos b\zeta\Biggl(\int^{\zeta}_{\tau}(t-\tau)^{\alpha-1}(\zeta-t)^{-\alpha}dt\Biggr)d\zeta\Biggr)d\tau\\
 &=C(\alpha)\int^{0}_{-\infty}e^{|y|^2 \tau}F_0 (\tau)
\Biggl(\int^{0}_{\tau}\cos b\zeta \,d\zeta\Biggr)d\tau
\end{align}
because making the change of variables $t=\zeta-(\zeta-\tau)s$ we see that
$$\int^{\zeta}_{\tau}(t-\tau)^{\alpha-1}(\zeta-t)^{-\alpha}dt=\int^{1}_{0}(1-s)^{\alpha-1}s^{-\alpha}ds=C(\alpha).$$
The calculation \eqref{6.22} is justified by Fubini's theorem and the fact that if we replace $\cos b\zeta$ and $g(t)$ with $|\cos b\zeta|$ and
$$g_0 (t)=\int^{t}_{-\infty}(t-\tau)^{\alpha-1}e^{|y|^2 \tau}|F_0 (\tau)|\,d\tau$$
respectively in the above calculation we get by Fubini's theorem for nonnegative functions that
\begin{align*}
 &\int^{0}_{-\infty}\int^{0}_{t}(\zeta-t)^{-\alpha}|\cos b\zeta|\,d\zeta\, g_0(t)\,dt\\
 &\leq C(\alpha)\int^{0}_{-\infty}e^{|y|^2 \tau}|F_0 (\tau)|
\Biggl(\int^{0}_{\tau}|\cos b\zeta|\,d\zeta\Biggr)d\tau\\
 &\leq C(\alpha)\int^{0}_{-\infty}(-\tau)e^{|y|^2 \tau}|F_0 (\tau)|\,d\tau<\infty
\end{align*}
by \eqref{6.20}

It follows now from \eqref{6.21}, \eqref{6.22} and \eqref{6.19} that 
$$0=\int^{0}_{-\infty}e^{|y|^2 \tau}F(y,\tau)\sin b\tau \,d\tau$$
for all $y\in\mathbb{R}^n \backslash\{0\}$ and all $b\in\mathbb{R}$.  Thus since the Fourier sine transform is one to one on $L^1 (-\infty,0)$ we have $F(y,\cdot)=0$ in $L^1 (-\infty,0)$ for all $y\in\mathbb{R}^n \backslash\{0\}$.  Hence by 
Fubini's theorem, $F=0$ in $L^1 (\mathbb{R}^n \times(-\infty,0))$,
which together with \eqref{6.16} and \eqref{6.14} implies \eqref{6.12}.\\
\medskip

\noindent\underline{Case II.}  Suppose \eqref{2.12}$_1$ holds.  Let
$f_T =f\raisebox{2pt}{$\chi$}_{\mathbb{R}^n \times\mathbb{R}_T}$ and $u=J_\alpha f_T$.
Then by \eqref{2.13} we have
$$f_T \in L^p (\mathbb{R}^n \times\mathbb{R}),$$
and by \eqref{2.4} and \eqref{6.11} we have
\begin{equation}\label{6.23}
 u=0\quad\text{in } \mathbb{R}^n \times\mathbb{R}_T .
\end{equation}
Let $J^{-\alpha}_{\varepsilon}u$ be as defined in Theorem \ref{thmA.1}.  By \eqref{6.23} we have for $l>\alpha$ that $(\Delta^{l}_{y,\tau}u)(x,t)=0$ for $(x,t)\in\mathbb{R}^n \times\mathbb{R}_T$ and $(y,\tau)\in\mathbb{R}^n \times(0,\infty)$.  Thus for 
$\varepsilon>0$ we have
$$J^{-\alpha}_{\varepsilon}u=0\quad\text{in } \mathbb{R}^n \times\mathbb{R}_T .$$
Hence \eqref{6.12} follows from Theorem \ref{thmA.1}. 
\end{proof}

\begin{proof}[Proof of Theorem \ref{thm2.4}]
 For $a,\tau>0$ and $\delta\geq0$ we have
 \begin{align}\label{L1}
  \notag \int_{|\xi|>\delta}\Phi_{\alpha,a,1}(\xi,\tau)\,d\xi
  \notag &=\int_{|\xi|>\delta}\frac{\tau^{\alpha-1}}{\Gamma(\alpha)}\,\frac{1}{(4\pi a^2 \tau)^{n/2}}e^{-\frac{|\xi|^2}{4a^2 \tau}}\,d\xi\\
  &=\frac{\tau^{\alpha-1}}{\Gamma(\alpha)}\,\frac{1}{\pi^{n/2}}\int_{|\eta|>\frac{\delta}{\sqrt{4a^2 \tau}}}e^{-|\eta|^2}d\eta.
 \end{align}
 In particular, taking $\delta=0$ we find that
 \begin{equation}\label{L2}
  \int_{\mathbb{R}^n}\Phi_{\alpha,a,1}(\xi,\tau)\,d\xi=\frac{\tau^{\alpha-1}}{\Gamma(\alpha)}\quad\text{ for }a,\tau>0.
 \end{equation}
 Let $\Omega$ be a compact subset of $\mathbb{R}^n \times\mathbb{R}$.  Choose $T>0$ such that
 \begin{equation}\label{L3}
  f=0\quad\text{ on }\mathbb{R}^n \times\mathbb{R}_{-T}
 \end{equation}
 and
 \begin{equation}\label{L4}
  \Omega\subset\mathbb{R}^n \times\mathbb{R}_T .
 \end{equation}
 Let $\varepsilon>0$.  Since $f$ is uniformly continuous on $\mathbb{R}^n \times\mathbb{R}$ there exists $\delta>0$ such that
 \begin{equation}\label{L5}
  |f(x-\xi,\zeta)-f(x,\zeta)|<\varepsilon
 \end{equation}
 whenever $x,\xi\in\mathbb{R}^n ,\zeta\in\mathbb{R}$, and $|\xi|<\delta$.
 
 Let $(x,t)\in\Omega$.  Then $t<T$ and thus for $\tau\geq2T$ we have
 $$t-\tau<T-2T=-T.$$
 Hence for $a>0$ we have by \eqref{L3} and \eqref{L2} that
 \begin{align}\label{L6}
  \notag &|(J_{\alpha,a,1}f-J_{\alpha,0,1}f)(x,t)|\\
  \notag &\leq\int^{2T}_{0}\left|\int_{\mathbb{R}^n}\Phi_{\alpha,a,1}(\xi,\tau)f(x-\xi,t-\tau)\,d\xi-\frac{\tau^{\alpha-1}}{\Gamma(\alpha)}f(x,t-\tau)\right|d\tau\\
  \notag &=\int^{2T}_{0}\left|\int_{\mathbb{R}^n}\Phi_{\alpha,a,1}(\xi,\tau)(f(x-\xi,t-\tau)-f(x,t-\tau))\,d\xi\right|d\tau\\
  &\leq K_1 (x,t)+K_2 (x,t)
 \end{align}
 where
 $$K_1 (x,t)=\int^{2T}_{0}\int_{|\xi|<\delta}\Phi_{\alpha,a,1}(\xi,\tau)|f(x-\xi,t-\tau)-f(x,t-\tau)|\,d\xi\,d\tau$$
 and
 $$K_2 (x,t)=\int^{2T}_{0}\int_{|\xi|>\delta}\Phi_{\alpha,a,1}(\xi,\tau)|f(x-\xi,t-\tau)-f(x,t-\tau)|\,d\xi\,d\tau.$$
 From \eqref{L5} and \eqref{L2} we conclude that
 $$K_1 (x,t)\leq\varepsilon\int^{2T}_{0}\left(\int_{\mathbb{R}^n}\Phi_{\alpha,a,1}(\xi,\tau)\,d\xi\right)d\tau=\varepsilon\int^{2T}_{0}\frac{\tau^{\alpha-1}}{\Gamma(\alpha)}d\tau$$
 and letting $M=2\| f\|_{L^\infty (\mathbb{R}^n \times\mathbb{R})}$ and using \eqref{L1} we obtain
 $$K_2 (x,t)\leq M\int^{2T}_{0}\left(\int_{|\xi|>\delta}
\Phi_{\alpha,a,1}(\xi,\tau)\,d\xi\right)d\tau\leq M\left(\int^{2T}_{0}\frac{\tau^{\alpha-1}}{\Gamma(\alpha)}d\tau\right)C(n,a,\delta,T)$$
 where
 $$C(n,a,\delta,T)=\frac{1}{\pi^{n/2}}\int_{|\eta|>\frac{\delta}{\sqrt{8a^2 T}}}e^{-|\eta|^2}d\eta\to0\quad\text{ as }a\to0^+ .$$
 The theorem therefore follows from \eqref{L6}.
\end{proof}

\begin{proof}[Proof of Theorem \ref{thm2.5}]
 For $b>0,\delta>0$, and $\xi\in\mathbb{R}^n \backslash\{0\}$ we have
 \begin{align}
  \notag \int^{\infty}_{\delta}\Phi_{\alpha,1,b}(\xi,\tau)\,d\tau
   &=\int^{\infty}_{\delta}
\frac{(\tau/b)^{\alpha-1}}{\Gamma(\alpha)}\,\frac{1}{(4\pi\tau/b)^{n/2}}e^{-\frac{b|\xi|^2}{4\tau}}d\tau/b\\
  \notag &=\int^{\infty}_{\delta}\frac{1}{\Gamma(\alpha)(4\pi)^{n/2}}\left(\frac{\tau}{b}\right)^{\alpha-1-n/2}e^{-\frac{b|\xi|^2}{4\tau}}\frac{1}{b}\,d\tau\\
  \notag &=\int^{\frac{b|\xi|^2}{4\delta}}_{0}\frac{1}{\Gamma(\alpha)(4\pi)^{n/2}}\left(\frac{|\xi|^2}{4\zeta}\right)^{\alpha-1-n/2}e^{-\zeta}\frac{|\xi|^2}{4\zeta^2}\,d\zeta\\
  \notag &=\frac{(|\xi|^2 /4)^{\alpha-n/2}}{\Gamma(\alpha)(4\pi)^{n/2}}\int^{\frac{b|\xi|^2}{4\delta}}_{0}\zeta^{n/2-\alpha-1}e^{-\zeta}d\zeta\\
  \label{R1} &=\frac{|\xi|^{2\alpha-n}}{4^\alpha \pi^{n/2}\Gamma(\alpha)}\int^{\frac{b|\xi|^2}{4\delta}}_0\zeta^{n/2-\alpha-1}e^{-\zeta}d\zeta\\
  \label{R2} &\leq\left(\frac{|\xi|^{2\alpha-n}}{4^\alpha \pi^{n/2}\Gamma(\alpha)}\right)\frac{1}{n/2-\alpha}\left(\frac{b|\xi|^2}{4\delta}\right)^{n/2-\alpha}
  =C(n,\alpha)\left(\frac{b}{\delta}\right)^{n/2-\alpha}.
 \end{align}
 Moreover, letting $\delta\to0^+$ in \eqref{R1} we obtain
 \begin{equation}\label{R3}
  \int^{\infty}_{0}\Phi_{\alpha,1,b}(\xi,\tau)\,d\tau=\frac{|\xi|^{2\alpha-n}}{\gamma(n,\alpha)}\quad\text{ for }b>0\text{ and }\xi\neq0,
 \end{equation}
 where $\gamma$ is given in \eqref{R0}.
 
 Let $\Omega$ be a compact subset of $\mathbb{R}^n \times\mathbb{R}$.  Choose $R>0$ such that
 \begin{equation}\label{R4}
  f=0\quad\text{ on }(\mathbb{R}^n \backslash B_R (0))\times\mathbb{R}
 \end{equation}
 and
 \begin{equation}\label{R5}
  \Omega\subset B_R (0)\times\mathbb{R}.
 \end{equation}
 Let $\varepsilon>0$.  Since $f$ is uniformly continuous on $\mathbb{R}^n \times\mathbb{R}$ there exists $\delta>0$ such that
 \begin{equation}\label{R6}
  |f(\eta,t-\tau)-f(\eta,t)|<\varepsilon
 \end{equation}
 whenever $\eta\in\mathbb{R}^n$, $t,\tau\in\mathbb{R}$, and $|\tau|<\delta$.
 
 Let $(x,t)\in\Omega$.  Then $|x|<R$ and thus for $|\xi|\geq2R$ we have
 $$|x-\xi|\geq|\xi|-|x|>2R-R=R.$$
 Hence for $b>0$ we find by \eqref{R4} and \eqref{R3} that
 \begin{align}\label{R7}
  \notag &|(J_{\alpha,1,b}f-J_{\alpha,1,0}f)(x,t)|\\
  \notag &\le \int_{|\xi|<2R}\left|\int^{\infty}_{0}\Phi_{\alpha,1,b}(\xi,\tau)f(x-\xi,t-\tau)\,d\tau-\frac{f(x-\xi,t)}{\gamma(n,\alpha)|\xi|^{n-2\alpha}}\right|d\xi\\
  \notag &=\int_{|\xi|<2R}\left|\int^{\infty}_{0}\Phi_{\alpha,1,b}(\xi,\tau)(f(x-\xi,t-\tau)-f(x-\xi,t))\,d\tau)\right|d\xi\\
  &\leq K_1 (x,t)+K_2 (x,t)
 \end{align}
 where
 $$K_1 (x,t)=\int_{|\xi|<2R}\int^{\delta}_{0}\Phi_{\alpha,1,b}(\xi,\tau)|f(x-\xi,t-\tau)-f(x-\xi,t)|\,d\tau \,d\xi$$
 and
 $$K_2 (x,t)=\int_{|\xi|<2R}\int^{\infty}_{\delta}\Phi_{\alpha,1,b}(\xi,\tau)|f(x-\xi,t-\tau)-f(x-\xi,t)|\,d\tau \,d\xi.$$
 From \eqref{R6} and \eqref{R3} we conclude
 $$K_1 (x,t)\leq\varepsilon\int_{|\xi|<2R}\left(\int^{\infty}_{0}\Phi_{\alpha,1,b}(\xi,\tau)\,d\tau\right)d\xi=\varepsilon\int_{|\xi|<2R}\frac{d\xi}{\gamma(n,\alpha)|\xi|^{n-2\alpha}}$$
 and letting $M=2\| f\|_{L^\infty (\mathbb{R}^n \times\mathbb{R})}$ and using \eqref{R2} we obtain
 \begin{align*}
 K_2 (x,t)&\leq M\int_{|\xi|<2R}\left(\int^{\infty}_{\delta}\Phi_{\alpha,1,b}(\xi,\tau)\,d\tau\right)d\xi\\
 &\leq MC(n,\alpha)\left(\frac{b}{\delta}\right)^{n/2-\alpha}|B_{2R}(0)|\to0\quad\text{ as }b\to0^+.
 \end{align*}
 The theorem therefore follows from \eqref{R7}.
\end{proof}

\section{Preliminary results for $J_\alpha$ problems}\label{sec7}
In this section we provide some lemmas needed for the proofs of our
results in Section \ref{sec4} dealing with solutions of the $J_\alpha$
problem \eqref{4.4}--\eqref{4.7}.

Let $\Omega=\mathbb{R}^n \times(a,b)$ where $n\geq1$ and $a<b$.  
Lemmas \ref{lem7.1} and \ref{lem7.2}
 give estimates for the convolution
\begin{equation}\label{7.1}
 (V_{\alpha,\Omega}f)(x,t)=\iint_{\Omega}\Phi_\alpha (x-\xi,t-\tau)f(\xi,\tau)\, d\xi \, d\tau
\end{equation}
where $\alpha>0$ and $\Phi_\alpha$ is defined in \eqref{2.2}.

\begin{rem}\label{rem7.1}
 Note that if $f:\mathbb{R}^n \times\mathbb{R}\to\mathbb{R}$ is a nonnegative measurable function such that $\| f\|_{L^\infty (\mathbb{R}^n \times\mathbb{R}_a)}=0$ then
 $$V_{\alpha,\Omega}f=J_\alpha f\quad\text{in } \Omega:=\mathbb{R}^n \times(a,b).$$
\end{rem}

\begin{lem}\label{lem7.1}
 For $\alpha>0,\,\Omega=\mathbb{R}^n \times(a,b)$ and $f\in L^\infty (\Omega)$ we have
 $$\| V_{\alpha,\Omega}f\|_{L^\infty (\Omega)}\leq\frac{(b-a)^\alpha}{\Gamma(\alpha+1)}\| f\|_{L^\infty (\Omega)}.$$
\end{lem}

\begin{proof}
 The lemma is obvious if $\| f\|_{L^\infty (\Omega)}=0$.  Hence we can assume $\| f\|_{L^\infty (\Omega)}>0$.  Then for $(x,t)\in\Omega$
 \begin{align*}
  \frac{|(V_{\alpha,\Omega}f)(x,t)|}{\| f\|_{L^\infty (\Omega)}}&\leq\int^{t}_{a}\frac{(t-\tau)^{\alpha-1}}{\Gamma(\alpha)}\overbrace{\Biggl(\int_{\xi\in\mathbb{R}^n}\Phi_1 (x-\xi,t-\tau)d\xi\Biggr)}^{=1}d\tau\\
  &=-\frac{(t-\tau)^\alpha}{\Gamma(\alpha+1)}\Biggr|^{\tau=t}_{\tau=a}=\frac{(t-a)^\alpha}{\Gamma(\alpha+1)}\\
  &\leq\frac{(b-a)^\alpha}{\Gamma(\alpha+1)}.
 \end{align*}
\end{proof}

\begin{lem}\label{lem7.2}
 Let $p,q\in[1,\infty]$, $\alpha$, and $\delta$ satisfy
 \begin{equation}\label{7.2}
  0\leq\delta:=\frac{1}{p}-\frac{1}{q}<\frac{2\alpha}{n+2}<1.
 \end{equation}
 Then $V_{\alpha,\Omega}$ maps $L^p (\Omega)$ continuously into $L^q (\Omega)$ and for $f\in L^p (\Omega)$ we have
 $$\| V_{\alpha,\Omega}f\|_{L^q (\Omega)}\leq M\| f\|_{L^p (\Omega)}$$
 where 
 $$M=C(b-a)^{\frac{2\alpha-(n+2)\delta}{2}}\text{ for some constant }C=C(n,\alpha,\delta).$$
\end{lem}

\begin{proof}
 Define $r\in[1,\infty)$ by
 \begin{equation}\label{7.3}
  1-\frac{1}{r}=\delta
 \end{equation}
 and define $P_\alpha ,\bar f:\mathbb{R}^n \times\mathbb{R}\to\mathbb{R}$ by
 \[
P_\alpha (x,t)=\Phi_\alpha (x,t)\raisebox{2pt}{$\chi$}_{(0,b-a)}(t)
\]
 and
 $$\bar f(x,t)=
 \begin{cases}
  f(x,t) & \text{if }(x,t)\in\Omega\\
  0 & \text{elsewhere.}
 \end{cases}$$
 Since for $t\in(a,b)$ and $\tau\in(a,t)$ we have $t-\tau\in(0,b-a)$ we see for $(x,t)\in\Omega$ that
 \begin{align}\label{7.4}
  \notag V_{\alpha,\Omega}f(x,t)&=\int^{t}_{a}\int_{\xi\in\mathbb{R}^n}P_\alpha (x-\xi,t-\tau)f(\xi,\tau)\, d\xi \, d\tau\\
  \notag &=\iint_\Omega P_\alpha (x-\xi,t-\tau)f(\xi,\tau)\, d\xi \, d\tau\\
  &=(P_\alpha *\bar f)(x,t)
 \end{align}
 where $*$ is the convolution operation in $\mathbb{R}^n \times\mathbb{R}$.  
 
 Also since
 $$\int_{\mathbb{R}^n}e^{-r|x|^2 /(4t)}dx=\left(\frac{4\pi t}{r}\right)^{n/2}$$
 we have by \eqref{7.2} and \eqref{7.3} that
 \begin{align*}
  \| P_\alpha \|_{L^r (\mathbb{R}^n \times\mathbb{R})}&=\frac{1}{\Gamma(\alpha)(4\pi)^{n/2}}\biggl(\int^{b-a}_{0}t^{r(\alpha-1-n/2)}\biggl(\int_{x\in\mathbb{R}^n}e^{-r|x|^2 /(4t)}dx\Biggr)dt\Biggr)^{1/r}\\
  &=C(n,\alpha,r)\Biggl(\int^{b-a}_{0}t^{r(\alpha-1-n/2)+\frac{n}{2}}dt\Biggr)^{1/r}\\
  &=C(n,\alpha,r)(b-a)^{\frac{2\alpha-(n+2)\delta}{2}}.
 \end{align*}
 Thus by \eqref{7.4}, \eqref{7.2}, \eqref{7.3}, and Young's inequality we have
 \begin{align*}
  \| V_{\alpha,\Omega} f\|_{L^q (\Omega)}&=\| P_\alpha *\bar f\|_{L^q (\Omega)}\leq\| P_\alpha *\bar f\|_{L^q (\mathbb{R}^n \times\mathbb{R})}\\
  &\leq\| P_\alpha \|_{L^r (\mathbb{R}^n \times\mathbb{R})}\|\bar f\|_{L^p (\mathbb{R}^n \times\mathbb{R})}\\
  &\leq C(b-a)^{\frac{2\alpha-(n+2)\delta}{2}}\| f\|_{L^p (\Omega)}.
 \end{align*}
\end{proof}

\begin{lem}\label{lem7.3}
 Suppose $f,\,p$, and $ K$ satisfy \eqref{4.4}--\eqref{4.7} and $(\lambda,\alpha)\in A\cup B$.  Then
 $$f\in X^\infty .$$
\end{lem}

\begin{proof}
 Let $T>0$ be fixed.  Then $f\in L^p (\mathbb{R}^n \times \mathbb{R}_T )$
 and to complete the proof it suffices to show
 \begin{equation}\label{7.5}
  f\in L^\infty (\mathbb{R}^n \times(0,T)).
 \end{equation}
 We consider two cases.\\
 \medskip
 
 \noindent\underline{Case I.}  Suppose $0<\alpha<\frac{n+2}{2p}$.  Then
 $$0<\lambda<\frac{n+2}{n+2-2\alpha p}$$
 and thus there exists $\varepsilon=\varepsilon(n,\lambda,\alpha,p)>0$ such that
 \[
\varepsilon<2\alpha p,\qquad 2\varepsilon<n+2-2\alpha p, \quad\text{ and }\quad
\lambda<\frac{n+2}{n+2-2\alpha p+2\varepsilon}.
\]
 Suppose
 \begin{equation}\label{7.6}
  f\in L^{p_0}(\mathbb{R}^n \times(0,T))\quad
\text{ for some }p_0 \in\left[p,\frac{n+2}{2\alpha}\right).
 \end{equation}
Then letting 
\[
q=\frac{(n+2)p_0}{n+2-2\alpha p_0 +\varepsilon}
\]
we have 
\[
\frac{1}{p_0}-\frac{1}{q}=\frac{2\alpha}{n+2}-\frac{\varepsilon}{(n+2)p_0}
\in \left(0,\frac{2\alpha}{n+2}\right).
\]
Hence by \eqref{4.6}, Remark \ref{rem7.1}, and Lemma \ref{lem7.2} we see that
 \[J_\alpha f\in L^q (\mathbb{R}^n \times(0,T)).\]
 Thus by \eqref{4.5} we find that
 \begin{equation}\label{7.7}
  0\leq f\leq K(J_\alpha f)^\lambda \in L^{q/\lambda}(\mathbb{R}^n \times(0,T)).
 \end{equation}
 Since
 \begin{align*}
  \frac{q/\lambda}{p_0}&=\frac{n+2}{\lambda(n+2-2\alpha p_0 +\varepsilon)}\geq\frac{n+2-2\alpha p+2\varepsilon}{n+2-2\alpha p_0 +\varepsilon}\\
  &\geq\frac{n+2-2\alpha p+2\varepsilon}{n+2-2\alpha p+\varepsilon}=C(n,\lambda,\alpha,p)>1
 \end{align*}
 we see that starting with $p_0 =p$ and iterating a finite number of times the process of going from \eqref{7.6} to \eqref{7.7} yields
 $$f\in L^{p_0}(\mathbb{R}^n \times(0,T))\quad\text{ for some }p_0 >\frac{n+2}{2\alpha}.$$
 Hence \eqref{7.5} follows from \eqref{4.5} and Lemma \ref{lem7.2}.\\
 \medskip
 
 \noindent\underline{Case II.}  Suppose $\alpha\geq\frac{n+2}{2p}$.  Clearly there exists $\widehat{\alpha}\in(0,\frac{n+2}{2p})$ such that $(\lambda,\widehat{\alpha})\in A\cup B$.  Then for $(x,t),(\xi,\tau)\in\mathbb{R}^n \times(0,T)$ we have
 \begin{align*}
  \frac{\Phi_\alpha (x-\xi,t-\tau)}{\Phi_{\widehat{\alpha}}(x-\xi,t-\tau)}&=(t-\tau)^{\alpha-\widehat{\alpha}}\Gamma(\widehat{\alpha})/\Gamma(\alpha)\\
  &\leq T^{\alpha-\widehat{\alpha}}\Gamma(\widehat{\alpha})/\Gamma(\alpha)\\
  &=C(T,\alpha,\widehat{\alpha}).
 \end{align*}
 Thus for $(x,t)\in\mathbb{R}^n \times(0,T)$ we have
 $$J_\alpha f(x,t)\leq C(T,\alpha,\widehat{\alpha})J_{\widehat{\alpha}}f(x,t)$$
 and hence by \eqref{4.5} we see that
 $$0\leq f\leq K C(T,\alpha,\widehat{\alpha})^\lambda
 (J_{\widehat{\alpha}}f)^\lambda \quad\text{almost everywhere in }\mathbb{R}^n \times(0,T).$$
 It follows therefore from Case I that $f$ satisfies \eqref{7.5}.
\end{proof}

\begin{lem}\label{lem7.4}
 Suppose $x\in\mathbb{R}^n$ and $t,\tau\in(0,\infty)$ satisfy
 \begin{equation}\label{7.8}
  |x|^2 <t \quad\text{and}\quad  \frac{t}{4}<\tau<\frac{3t}{4}.
 \end{equation}
 Then
 $$\int_{|\xi|^2 <\tau}\Phi_1 (x-\xi,t-\tau)\,d\xi\geq C(n)>0$$
 where $\Phi_\alpha$ is defined by \eqref{2.2}.
\end{lem}

\begin{proof}
 Making the change of variables $z=\frac{x-\xi}{\sqrt{4(t-\tau)}}$, letting $e_1 =(1,0,...,0)$, and using \eqref{7.8} and \eqref{2.2} we find that
 \begin{align*}
  \int_{|\xi|^2 <\tau}\Phi_1 (x-\xi,t-\tau)\,d\xi
  &=\frac{1}{\pi^{n/2}}\int_{|z-\frac{x}{\sqrt{4(t-\tau)}}|<\frac{\sqrt{\tau}}{\sqrt{4(t-\tau)}}}e^{-|z|^2}dz\\
  &\geq\frac{1}{\pi^{n/2}}\int_{|z-\frac{\sqrt{t}}{\sqrt{4(t-\tau)}}e_1 |<\frac{\sqrt{\tau}}{\sqrt{4(t-\tau)}}}e^{-|z|^2}dz\\
  &\geq\frac{1}{\pi^{n/2}}\int_{|z-e_1 |<\frac{1}{2\sqrt{3}}}e^{-|z|^2}dz\\
  &=C(n)>0
 \end{align*}
 where in this calculation we used the fact that the integral of $e^{-|z|^2}$ over a ball is decreased if the absolute value of the center of the ball is increased or the radius of the ball is decreased.
\end{proof}

\begin{lem}\label{lem7.5}
 For $\tau<t\leq T$ and $|x|\leq\sqrt{T-t}$ we have
 $$\int_{|\xi|<\sqrt{T-\tau}}\Phi_1(x-\xi,t-\tau)\,d\xi\geq C$$
 where $C=C(n)$ is a positive constant. 
\end{lem}

\begin{proof}
 Making the change of variables $z=\frac{x-\xi}{\sqrt{t-\tau}}$ 
and letting $e_1 =(1,0,...,0)$ we get
 \begin{align}
  \notag
   \int_{|\xi|<\sqrt{T-\tau}}\Phi_1(x-\xi,t-\tau)\,d\xi&=\frac{1}{(4\pi)^{n/2}}
\frac{1}{(t-\tau)^{n/2}}\int_{|\xi|<\sqrt{T-\tau}}e^{-\frac{|x-\xi|^2}{4(t-\tau)}}d\xi\notag\\
  &=\frac{1}{(4\pi)^{n/2}}\int_{|z-\frac{x}{\sqrt{t-\tau}}|<\frac{\sqrt{T-\tau}}{\sqrt{t-\tau}}}
e^{-|z|^2/4}dz \label{7.9}\\
  &\geq\frac{1}{(4\pi)^{n/2}}\int_{|z-\frac{\sqrt{T-\tau}}{\sqrt{t-\tau}}e_1|<\frac{\sqrt{T-\tau}}{\sqrt{t-\tau}}}e^{-|z|^2/4}dz
\label{7.10}\\
  &\geq\frac{1}{(4\pi)^{n/2}}\int_{|z-e_1
    |<1}e^{-|z|^2/4}dz, \label{7.11}
 \end{align}
 where the last two inequalities need some explanation.  Since
 $|x|\leq\sqrt{T-t}<\sqrt{T-\tau}$, the center of the ball of integration
 in \eqref{7.9} is closer to the origin than the center of the ball
 of integration in \eqref{7.10}.  Thus, since the integrand is a
 decreasing function of $|z|$, we obtain \eqref{7.10}.  Since
 $\sqrt{T-\tau}\geq\sqrt{t-\tau}$, the ball of integration in \eqref{7.10}
 contains the ball of integration in \eqref{7.11} and hence
 \eqref{7.11} holds.
\end{proof}

\begin{lem}\label{lem7.6}
 Suppose $\alpha>0$, $\gamma>0$, $p\geq1$, and 
 \[f_0
 (x,t)=\left(\frac{1}{t}\right)^{\frac{n+2}{2p}-\gamma}\raisebox{2pt}{$\chi$}_{\Omega_0}(x,t)\quad\text{
   where }\Omega_0 =\{(x,t)\in \mathbb{R}^n\times\mathbb{R}:|x|^2 <t\}.\]
 Then $f_0 \in X^p$ and
 \[C_1 \left(\frac{1}{t}\right)^{\frac{n+2}{2p}-\gamma-\alpha}\leq J_\alpha f_0 (x,t)\leq C_2 \left(\frac{1}{t}\right)^{\frac{n+2}{2p}-\gamma-\alpha}\quad\text{for }(x,t)\in\Omega_0\]
 where $C_1$ and $C_2$ are positive constants depending only on $n,\alpha,\gamma$, and $p$.
\end{lem}

\begin{proof}
 For $T>0$ we have
 \begin{align*}
  \| f_0 \|^{p}_{L^p(\mathbb{R}^n \times\mathbb{R}_T )}&=\int^{T}_{0}\int_{|x|<\sqrt{t}}\left(\frac{1}{t}\right)^{\frac{n+2}{2}-\gamma p}\, dx \, dt\\
  &=C(n)\int^{T}_{0}t^{\gamma p-1}dt<\infty
 \end{align*}
because $\gamma p>0$. Hence $f_0\in X^p$.  

Also for $(x,t)\in \mathbb{R}^n\times (0,\infty)$  we have 
\begin{align}
  \notag J_\alpha f_0 (x,t)&=\int^{t}_{-\infty}\int_{\xi\in\mathbb{R}^n}\Phi_\alpha (x-\xi,t-\tau)f_0 (\xi,\tau)\, d\xi \, d\tau\\
\label{7.12}&=\frac{1}{\Gamma(\alpha)}\int^{t}_{0}(t-\tau)^{\alpha-1}\left(\frac{1}{\tau}\right)^{\frac{n+2}{2p}-\gamma}\biggl(\int_{|\xi|^2 <\tau}\Phi_1(x-\xi,t-\tau)\,d\xi\Biggr)d\tau.
\end{align}
Hence by Lemma \ref{lem7.4} we see for $(x,t)\in\Omega_0$ that 
\begin{align*}
  J_\alpha f_0 (x,t)
&\geq C(n,\alpha)\int^{3t/4}_{t/4}(t-\tau)^{\alpha-1}\left(\frac{1}{\tau}\right)^{\frac{n+2}{2p}-\gamma}d\tau\\
&=C(n,\alpha)t^{\alpha-\frac{n+2}{2p}+\gamma}\int^{3/4}_{1/4}(1-s)^{\alpha-1}\left(\frac{1}{s}\right)^{\frac{n+2}{2p}-\gamma}ds\quad\text{where }\tau=ts\\
  &=C(n,\alpha,\gamma,p)t^{\alpha-\frac{n+2}{2p}+\gamma}.
 \end{align*}

 Moreover for $(x,t)\in\mathbb{R}^n \times(0,\infty)$ and $0<\tau<t/2$ we have
 \begin{align*}
   \int_{|\xi|^2 <\tau}\Phi_1
   (x-\xi,t-\tau)\,d\xi&=\frac{1}{\pi^{n/2}}
\int_{|z-\frac{x}{\sqrt{4(t-\tau)}}|<\frac{\sqrt{\tau}}{\sqrt{4(t-\tau)}}}e^{-|z|^2}dz
\quad\text{where }z=\frac{x-\xi}{\sqrt{4(t-\tau)}}\\
   &\leq\frac{|B_1(0)|}{\pi^{n/2}}
\Biggl(\frac{\sqrt{\tau}}{\sqrt{4(t-\tau)}}\Biggr)^n
 \end{align*}
 and for $(x,t)\in\mathbb{R}^n \times(0,\infty)$ and $t/2<\tau<t$ we have
 $$\int_{|\xi|^2 <\tau}\Phi_1 (x-\xi,t-\tau)\,d\xi\leq\int_{\mathbb{R}^n}\Phi_1 (x-\xi,t-\tau)\,d\xi=1.$$
 Thus by \eqref{7.12} for $(x,t)\in\mathbb{R}^n \times(0,\infty)$ we have
 \begin{align*}
  J_\alpha f_0 (x,t)&\leq C(n,\alpha)\Biggl[\int^{t/2}_{0}(t-\tau)^{\alpha-1}\left(\frac{1}{\tau}\right)^{\frac{n+2}{2p}-\gamma}\left(\frac{\tau}{t-\tau}\right)^{n/2}d\tau\\
  &\phantom{\leq C(n,\alpha)+}+\int^{t}_{t/2}(t-\tau)^{\alpha-1}\left(\frac{1}{\tau}\right)^{\frac{n+2}{2p}-\gamma}d\tau\Biggr]\\
  &=C(n,\alpha)t^{\alpha-\frac{n+2}{2p}+\gamma}\Biggl[\int^{1/2}_{0}(1-s)^{\alpha-1}\left(\frac{1}{s}\right)^{\frac{n+2}{2p}-\gamma}\left(\frac{s}{1-s}\right)^{n/2}ds\\
  &\phantom{=C(n,\alpha)t^{\alpha-\frac{n+2}{2p}+\gamma}+}+\int^{1}_{1/2}(1-s)^{\alpha-1}\left(\frac{1}{s}\right)^{\frac{n+2}{2p}-\gamma}ds\Biggr]\\
  &=C(n,\alpha,\gamma,p)t^{\alpha-\frac{n+2}{2p}+\gamma}
 \end{align*}
 because $\alpha$ and $\gamma$ are positive.
\end{proof}

\begin{lem}\label{lem7.7}
 Suppose $\alpha>0$, $\gamma\in\mathbb{R}$, $0\leq t_0 <T,\,p\in[1,\infty)$, and
 \[
f(x,t)=\left(\frac{1}{T-t}\right)^{\frac{n+2}{2p}-\gamma}\raisebox{2pt}{$\chi$}_\Omega(x,t)
\]
 where
 $$\Omega=\{(x,t)\in\mathbb{R}^n \times(t_0 ,T):|x|<\sqrt{T-t}\}.$$
 Then
 $$J_\alpha f(x,t)\geq C\left(\frac{1}{T-t}\right)^{\frac{n+2}{2p}-\gamma-\alpha}$$
 for $(x,t)\in\Omega^+ :=\{(x,t)\in\Omega:\frac{T+t_0}{2}<t<T\}$ where $C=C(n,\alpha,\gamma,p)>0$.  Moreover,
 \begin{equation}\label{7.13}
  f\in L^p (\mathbb{R}^n \times\mathbb{R})\text{ if and only if }\gamma>0
 \end{equation}
 and in this case
 \begin{equation}\label{7.14}
  \| f\|^{p}_{L^p (\mathbb{R}^n \times\mathbb{R})}=C(n)\int^{T-t_0}_{0}s^{\gamma p-1}ds.
 \end{equation}
\end{lem}

\begin{proof}
 Since
 \begin{align*}
  \| f\|^{p}_{L^p (\mathbb{R}^n \times\mathbb{R})}&=\int^{T}_{t_0}\int_{|x|<\sqrt{T-t}}(T-t)^{\gamma p-\frac{n+2}{2}}\, dx \, dt\\
  &=C(n)\int^{T}_{t_0}(T-t)^{\gamma p-1}dt=C(n)\int^{T-t_0}_{0}s^{\gamma p-1}ds
 \end{align*}
 we see that \eqref{7.13} and \eqref{7.14} hold.
 
 Let $r=\frac{n+2}{2p}-\gamma-\alpha$.  Then for $(x,t)\in\Omega$ we have
 \begin{align*}
  J_\alpha
   f(x,t)&=\int^{t}_{t_0}(T-\tau)^{-r-\alpha}\int_{|\xi|<\sqrt{T-\tau}}\Phi_\alpha
           (x-\xi,t-\tau)\, d\xi \, d\tau\\
  &=C\int^{t}_{t_0}(T-\tau)^{-r-\alpha}(t-\tau)^{\alpha-1}\Biggl(\int_{|\xi|<\sqrt{T-\tau}}\Phi_1 (x-\xi,t-\tau)\,d\xi\Biggr)d\tau\\
  &\geq C\int^{t}_{t_0}(T-\tau)^{-r-\alpha}(t-\tau)^{\alpha-1}d\tau,
\quad\text{ by Lemma \ref{lem7.5},}\\
  &=C(T-t)^{-r}g\biggl(\frac{t-t_0}{T-t}\biggr)
 \end{align*}
 where $g(z)=\int^{z}_{0}(\zeta+1)^{-r-\alpha}\zeta^{\alpha-1}d\zeta$ and where we made the change of variables $t-\tau=(T-t)\zeta$.  Thus
 $$J_\alpha f(x,t)\geq C(T-t)^{-r}\quad\text{for }(x,t)\in\Omega^+$$
 because $\frac{t-t_0}{T-t}>1$ in $\Omega^+$.
\end{proof}

\section{Proofs of results for $J_\alpha$ problems}\label{sec8}
In this section we prove our results stated in Section \ref{sec4}
concerning pointwise bounds for nonnegative solutions $f$ of
\eqref{4.4}--\eqref{4.7}. As explained in Section \ref{sec4}, these
results immediately imply Theorems \ref{thm3.1}--\ref{thm3.6} in
Section \ref{sec3}.

\begin{rem}\label{rem8.1}
 The function $g:\mathbb{R}^n \times\mathbb{R}\to[0,\infty)$ defined by
 $$g(x,t)=g(t)=
 \begin{cases}
  (Mt^\alpha )^{\frac{\lambda}{1-\lambda}} & \text{for }t>0\\
  0 & \text{for }t\leq0,
 \end{cases}$$
 where $\alpha>0$, $0<\lambda<1$, and $M=M(\alpha,\lambda)$ 
is defined in \eqref{4.11}, satisfies
 \begin{equation}\label{8.1}
  g=(J_\alpha g)^\lambda \quad\text{in } \mathbb{R}^n \times\mathbb{R}
 \end{equation}
 which can be verified using \eqref{5.3}.  Even though $g\notin X^p$
 for all $p\geq1$, it will be useful in our analysis of solutions of
 \eqref{4.5}, \eqref{4.6} which are in $X^p$ for some $p\geq1$.
\end{rem}

\begin{rem}\label{rem8.2}
 It will be convenient to scale \eqref{4.5} as follows.  Suppose $
 K,\lambda,\alpha,T\in(0,\infty)$, $\lambda\neq1$, and $f,\bar f:\mathbb{R}^n \times\mathbb{R}\to\mathbb{R}$ are nonnegative measurable functions such that $f=\bar f=0$ in 
 $\mathbb{R}^n \times(-\infty,0)$ and
 $$f(x,t)= K^{\frac{1}{1-\lambda}}T^{\frac{\alpha\lambda}{1-\lambda}}\bar f(\bar x,\bar t)$$
 where
 $$x=T^{1/2}\bar x \quad\text{ and }\quad  t=T\bar t.$$
 Then $f$ satisfies
 $$0\leq f\leq K(J_\alpha f)^\lambda \quad\text{in } \mathbb{R}^n \times\mathbb{R}$$
 if and only if $\bar f$ satisfies
 $$0\leq\bar f\leq(J_\alpha \bar f)^\lambda \quad\text{in } \mathbb{R}^n \times\mathbb{R}.$$
 Moreover
 $$\frac{f(x,t)}{ K^{\frac{1}{1-\lambda}}t^{\frac{\alpha\lambda}{1-\lambda}}}=\frac{\bar f(\bar x,\bar t)}{\bar t^{\frac{\alpha\lambda}{1-\lambda}}}\quad\text{for }(x,t)\in\mathbb{R}^n \times(0,\infty)$$
 and
 $$\frac{J_\alpha f(x,t)}{ K^{\frac{1}{1-\lambda}}t^{\frac{\alpha}{1-\lambda}}}=\frac{J_\alpha \bar f(\bar x,\bar t)}{\bar t^{\frac{\alpha}{1-\lambda}}}\quad\text{for }(x,t)\in\mathbb{R}^n \times(0,\infty).$$
\end{rem}
\medskip

\begin{proof}[Proof of Theorem \ref{thm4.1}]
  Suppose for contradiction that \eqref{4.8} is false.  Then there
  exists $T>0$ such that
\[
\|f\|_{L^\infty(\mathbb{R}^n \times\mathbb{R}_T)}>0.
\]
Hence by \eqref{4.6} there exists $t_0\in[0,T)$ such that
\[
\| f\|_{L^\infty (\mathbb{R}^n \times\mathbb{R}_t )}
\begin{cases}
 =0 & \text{for }t\leq t_0\\
 >0 & \text{for }t>t_0 .
\end{cases}
\]
Thus by Remark \ref{rem7.1}, we have for all $b>t_0$ that 
$$J_\alpha f=V_{\alpha, \Omega_b}f\quad\text{in } \Omega_b$$
where $\Omega_b =\mathbb{R}^n \times (t_0 ,b)$ and $V_{\alpha,\Omega}$ is defined by \eqref{7.1}.  Also, by Lemma \ref{lem7.3},
$$\| f\|_{L^\infty (\Omega_b)}\leq\| f\|_{L^\infty (\Omega_T )}<\infty\quad\text{for }t_0 <b<T.$$
It follows therefore from \eqref{4.5} and Lemma \ref{lem7.1} that for $t_0 <b<T$ we have
$$ 0<K^{-1}\leq\frac{\| V_{\alpha, \Omega_b}f\|^{\lambda}_{L^\infty (\Omega_b)}}{\| f\|_{L^\infty (\Omega_b )}}\leq\Biggl(\frac{(b-t_0 )^\alpha}{\Gamma(\alpha+1)}\Biggr)^\lambda \| f\|^{\lambda-1}_{L^\infty (\Omega_b )}\to0\quad\text{ as }b\to t^{+}_{0}$$
because $\lambda\geq1$.  This contradiction proves Theorem
\ref{thm4.1}. 
\end{proof}

\begin{proof}[Proof of Theorem \ref{thm4.2}]
 By Remark \ref{rem8.2} with $T=1$ we can assume $ K=1$.  For $b>0$ we have by Lemma \ref{lem7.3} that 
$$f\in L^\infty (\mathbb{R}^n \times\mathbb{R}_b )$$
and by \eqref{4.5}, \eqref{4.6}, Remark \ref{rem7.1} with $a=0$, and 
Lemma \ref{lem7.1} that
$$\| f\|_{L^\infty (\Omega_b )}\leq\| J_\alpha f\|^{\lambda}_{L^\infty (\Omega_b )}\leq\Biggl(\frac{b^\alpha}{\Gamma(\alpha+1)}\| f\|_{L^\infty (\Omega_b )}\Biggr)^\lambda$$
where $\Omega_b =\mathbb{R}^n \times(0,b)$.  Thus, since $0<\lambda<1$, we see that
\begin{equation}\label{8.2}
 \| f\|_{L^\infty (\Omega_b
   )}\leq\Biggl(\frac{b^\alpha}{\Gamma(\alpha+1)}\Biggr)^{\frac{\lambda}{1-\lambda}}
\quad\text{for all }b>0.
\end{equation}
Define $\{\gamma_j \}\subset(0,\infty)$ by $\gamma_1 =1$ and
\begin{equation}\label{8.3}
 \gamma_{j+1}=(\bar M\gamma_j )^\lambda ,\,j=1,2,...,\quad\text{ where }\bar M=\Gamma(\alpha+1)M.
\end{equation}
Then, since $0<\lambda<1$, we see that
\begin{equation}\label{8.4}
 \gamma_j \to\bar M^{\frac{\lambda}{1-\lambda}}\quad\text{ as }j\to\infty.
\end{equation}
Suppose for some positive integer $j$ that
\begin{equation}\label{8.5}
 \| f\|_{L^\infty (\Omega_b )}\leq\gamma_j
 \Biggl(\frac{b^\alpha}{\Gamma(\alpha+1)}\Biggr)^{\frac{\lambda}{1-\lambda}}\quad\text{for
   all }b>0.
\end{equation}
Then for $b>0$ and $(x,t)\in\Omega_b$ we find from \eqref{4.5} and \eqref{5.3} that
\begin{align}\label{8.6}
 \notag f(x,t)&\leq(J_\alpha f(x,t))^\lambda \\
 \notag &\leq\Biggl(\int^{t}_{0}\frac{(t-\tau)^{\alpha-1}}{\Gamma(\alpha)}\Biggl(\int_{\xi\in\mathbb{R}^n}\Phi_1(x-\xi,t-\tau)\,d\xi\Biggr)\| f\|_{L^\infty (\Omega_\tau )}d\tau\Biggr)^\lambda \\
 \notag &\leq\Biggl(\int^{t}_{0}\frac{(t-\tau)^{\alpha-1}}{\Gamma(\alpha)}\gamma_j \Biggl(\frac{\tau^\alpha}{\Gamma(\alpha+1)}\Biggr)^{\frac{\lambda}{1-\lambda}}d\tau\Biggr)^\lambda \\
 \notag &=\Biggl(\gamma_j \frac{1}{\Gamma(\alpha)\Gamma(\alpha+1)^{\frac{\lambda}{1-\lambda}}}\int^{t}_{0}(t-\tau)^{\alpha-1}\tau^{\frac{\alpha\lambda}{1-\lambda}}d\tau\Biggr)^\lambda \\
 \notag &=\Biggl(\gamma_j \frac{\Gamma(\alpha)\Gamma(\frac{\alpha\lambda}{1-\lambda}+1)t^{\alpha+\frac{\alpha\lambda}{1-\lambda}}}{\Gamma(\alpha)\Gamma(\alpha+1)^{\frac{\lambda}{1-\lambda}}\Gamma(\alpha+\frac{\alpha\lambda}{1-\lambda}+1)}\Biggr)^\lambda \\
 \notag &=\Biggl(\gamma_j \frac{Mt^{\frac{\alpha}{1-\lambda}}}{\Gamma(\alpha+1)^{\frac{\lambda}{1-\lambda}}}\Biggr)^\lambda =\Biggl(\gamma_j \frac{\bar Mt^{\frac{\alpha}{1-\lambda}}}{\Gamma(\alpha+1)^{\frac{1}{1-\lambda}}}\Biggr)^\lambda \\
 &=\gamma_{j+1}\Biggl(\frac{t^\alpha}{\Gamma(\alpha+1)}\Biggr)^{\frac{\lambda}{1-\lambda}}.
\end{align}
Thus
$$\| f\|_{L^\infty (\Omega_b
  )}\leq\gamma_{j+1}\Biggl(\frac{b^\alpha}{\Gamma(\alpha+1)}\Biggr)^{\frac{\lambda}{1-\lambda}}\quad\text{for all }b>0.$$
Hence \eqref{4.9} follows inductively from \eqref{8.2}--\eqref{8.5}.

Finally, repeating the calculation \eqref{8.6} with $\gamma_j =\gamma_{j+1}=\bar M^{\frac{\lambda}{1-\lambda}}$ we get
$$(J_\alpha f(x,t))^\lambda \leq\bar M^{\frac{\lambda}{1-\lambda}}\Biggl(\frac{t^\alpha}{\Gamma(\alpha+1)}\Biggr)^{\frac{\lambda}{1-\lambda}}\quad\text{for }(x,t)\in\Omega_b$$
which proves \eqref{4.10}. 
\end{proof}

\begin{proof}[Proof of Theorem \ref{thm4.3}]
  By Remark \ref{rem8.2} we can assume $ K=T=1$.  For $(x,t)\in\mathbb{R}^n \times\mathbb{R}$ and $\delta\in(0,1)$ let 
\begin{equation}\label{8.7}
 g_\delta (x,t)=g_\delta (t)=\psi_\delta (t)g(t)
\end{equation}
where $g$ is as in Remark \ref{rem8.1} and $\psi_\delta \in C^\infty (\mathbb{R}\to[0,1])$ satisfies
$$\psi_\delta (t)=
\begin{cases}
 1 & \text{if }t\leq1\\
 0 & \text{if }t\geq1+\delta.
\end{cases}$$
Then for $1\leq t\leq1+\delta$
\begin{align*}
 J_\alpha g(t)-J_\alpha g_\delta (t)&=\int^{t}_{1}\frac{(t-\tau)^{\alpha-1}}{\Gamma(\alpha)}g(\tau)(1-\psi_\delta (\tau))\,d\tau\\
 &\leq\int^{t}_{1}\frac{(t-\tau)^{\alpha-1}}{\Gamma(\alpha)}g(\tau)\,d\tau\leq g(1+\delta)\int^{t}_{1}\frac{(t-\tau)^{\alpha-1}}{\Gamma(\alpha)}\,d\tau\\
 &=g(1+\delta)\frac{(t-1)^\alpha}{\Gamma(\alpha+1)}\leq g(2)\frac{\delta^\alpha}{\Gamma(\alpha+1)}
\end{align*}
and thus by \eqref{8.1} we have for $1\leq t\leq1+\delta$ that
\begin{align*}
 \frac{J_\alpha g_\delta (t)}{J_\alpha g(t)}&=\frac{J_\alpha g(t)-(J_\alpha g(t)-J_\alpha g_\delta (t))}{g(t)^{1/\lambda}}\\
 &\geq1-\frac{g(2)\delta^\alpha}{\Gamma(\alpha+1)g(1)^{1/\lambda}}\\
 &=1-C(\alpha,\lambda)\delta^\alpha \geq\sqrt{\frac{N}{M}}
\end{align*}
provided we choose $\delta=\delta(\alpha,\lambda,N)\in(0,1)$ sufficiently small.  Hence for $1\leq t\leq1+\delta$ we see from \eqref{8.1} that
\begin{equation}\label{8.8}
 g_\delta (t)\leq g(t)=(J_\alpha g(t))^\lambda \leq\left(\frac{M}{N}\right)^{\lambda/2}(J_\alpha g_\delta (t))^\lambda 
\end{equation}
which by \eqref{8.7} and \eqref{8.1} holds for all other $t$ as well.

Next let $\varphi(x)=e^{-\psi(x)}$ where $\psi(x)=\sqrt{1+|x|^2}-1$.  Then for $\varepsilon\in(0,1),\,\gamma>1$, and $|\xi-x|<\gamma\sqrt{2}$ we have
$$\frac{\varphi(\varepsilon\xi)}{\varphi(\varepsilon x)}=e^{-(\psi(\varepsilon\xi)-\psi(\varepsilon x))}\geq e^{-\varepsilon|\xi-x|}\geq e^{-\varepsilon\gamma\sqrt{2}}.$$
Thus defining $f_\varepsilon :\mathbb{R}^n \times\mathbb{R}\to[0,\infty)$ by
$$f_\varepsilon (x,t)=\varphi(\varepsilon x)\left(\frac{N}{M}\right)^{\frac{\lambda}{1-\lambda}}g_\delta (t)$$
we find for $|\xi-x|<\gamma\sqrt{2}$ and $\tau\in\mathbb{R}$ that
$$f_\varepsilon (\xi,\tau)\geq\varphi(\varepsilon x)e^{-\varepsilon\gamma\sqrt{2}}\left(\frac{N}{M}\right)^{\frac{\lambda}{1-\lambda}}g_\delta (\tau).$$
Thus for $(x,t)\in\mathbb{R}^n \times(0,2)$ we have
\begin{equation}\label{8.9}
 J_\alpha f_\varepsilon (x,t)\geq\varphi(\varepsilon x)e^{-\varepsilon\gamma\sqrt{2}}\left(\frac{N}{M}\right)^{\frac{\lambda}{1-\lambda}}\int^{t}_{0}\frac{(t-\tau)^{\alpha-1}}{\Gamma(\alpha)}g_\delta (\tau)\int_{|\xi-x|<\gamma\sqrt{2}}\Phi_1 (x-\xi,t-\tau)\, d\xi \, d\tau.
\end{equation}
But for $x,\xi\in\mathbb{R}^n$ and $0<\tau<t<2$ we find making the change of variables $z=\frac{x-\xi}{\sqrt{4(t-\tau)}}$ that
\begin{align*}
 \int_{|\xi-x|<\gamma\sqrt{2}}\Phi_1(x-\xi,t-\tau)\,d\xi&\geq\int_{|\xi-x|<\gamma\sqrt{t-\tau}}\frac{1}{(4\pi(t-\tau))^{n/2}}e^{-\frac{|x-\xi|^2}{4(t-\tau)}}d\xi\\
 &=\frac{1}{\pi^{n/2}}\int_{|z|<\gamma/2}e^{-|z|^2}dz=:I(\gamma)\to1
\end{align*}
as $\gamma\to\infty$.  Thus by \eqref{8.9} and \eqref{8.8} we have for $(x,t)\in\mathbb{R}^n \times(0,1+\delta)$ that
\begin{align}\label{8.10}
 \notag \frac{(J_\alpha f_\varepsilon (x,t))^\lambda}{f_\varepsilon (x,t)}&\geq\frac{\varphi(\varepsilon x)^\lambda e^{-\varepsilon\gamma\lambda\sqrt{2}}(\frac{N}{M})^{\frac{\lambda^2}{1-\lambda}}I(\gamma)^\lambda (J_\alpha g_\delta (t))^\lambda}{\varphi(\varepsilon x)(\frac{N}{M})^{\frac{\lambda}{1-\lambda}}g_\delta (t)}\\
 &\geq\left(\frac{M}{N}\right)^{\lambda/2}I(\gamma)^\lambda e^{-\varepsilon\gamma\lambda\sqrt{2}}.
\end{align}
So first choosing $\gamma$ so large that
$(\frac{M}{N})^{\lambda/2}I(\gamma)^\lambda >1$ and then choosing
$\varepsilon>0$ so small that \eqref{8.10} is greater than 1  we see
that $f:=f_\varepsilon$ satisfies\eqref{4.5} in
$\mathbb{R}^n \times(0,1+\delta)$.  Thus, since $g_\delta (t)$ and
hence $f(x,t)$ is identically zero in
$\mathbb{R}^n \times((-\infty,0]\cup[1+\delta,\infty))$ see that $f$
satisfies \eqref{4.5}, \eqref{4.6}.

From the exponential decay of $\varphi(x)$ as $|x|\to\infty$, we see that $f$ satisfies \eqref{4.12}.  Also since $f$ is uniformly continuous and bounded on $\mathbb{R}^n \times\mathbb{R}$ and 
$$\int^{b}_{a}\int_{\mathbb{R}^n}\Phi_\alpha (x,t)\, dx \, dt=\frac{1}{\Gamma(\alpha+1)}(b^\alpha -a^\alpha )\quad\text{for }a<b,$$
we easily check that \eqref{4.12.5} holds.

Finally, since
\[f(0,t)=\left(\frac{N}{M}\right)^{\frac{\lambda}{1-\lambda}}g(t)\quad\text{for }0\leq t\leq1\]
we find that \eqref{4.13} holds and thus \eqref{4.14} follows from
\eqref{4.5}. 
\end{proof}

\begin{proof}[Proof of Theorem \ref{thm4.4}]
  By Remark \ref{rem8.2} with $T=1$ we can assume $ K=1$.  Define $\bar f:\mathbb{R}^n \times\mathbb{R}\to[0,\infty)$ by
\begin{equation}\label{8.11}
 \bar f(x,t)=g(t)\raisebox{2pt}{$\chi$}_{\{|x|^2 <t\}}(x,t)
\end{equation}
where $g$ is defined in Remark \ref{rem8.1}.  Then for $(x,t)\in\mathbb{R}^n \times(0,\infty)$ we have
$$J_\alpha \bar f(x,t)=\int^{t}_{0}\frac{(t-\tau)^{\alpha-1}}{\Gamma(\alpha)}\Biggl(\int_{|\xi|^2 <\tau}\Phi_1 (x-\xi,t-\tau)\,d\xi\Biggr)g(\tau)\,d\tau.$$
Thus by Lemma \ref{lem7.4} we see for $|x|^2 <t$ that
\begin{align}\label{8.12}
 \notag J_\alpha \bar f(x,t)&\geq C(n,\alpha,\lambda)\int^{3t/4}_{t/4}(t-\tau)^{\alpha-1}\tau^{\frac{\alpha\lambda}{1-\lambda}}d\tau\\
 \notag &=C(n,\alpha,\lambda)t^{\frac{\alpha}{1-\lambda}}\\
 \notag &=C(n,\alpha,\lambda)g(x,t)^{1/\lambda}\\
 &=C(n,\alpha,\lambda)\bar f(x,t)^{1/\lambda}
\end{align}
which also holds in $(\mathbb{R}^n \times\mathbb{R})\backslash\{|x|^2\le t\}$ 
because $\bar f=0$ there.  Thus letting $f=L\bar f$ where
$$L=C^{\frac{\lambda}{1-\lambda}}$$
where $C=C(n,\alpha,\lambda)$ is as in \eqref{8.12} we find that $f$ satisfies \eqref{4.4}--\eqref{4.6}.

It follows from \eqref{8.11} and the definitions of $g$ and $f$ that
there exists $N>0$ such that \eqref{4.15} holds.  Thus, since $f$
solves \eqref{4.5} we obtain \eqref{4.16}. 
\end{proof}

\begin{proof}[Proof of Theorem \ref{thm4.5}]
 Since $|R_j|<\infty$, to prove Theorem \ref{thm4.5} it suffices to
 show for each $\varepsilon\in(0,1)$ that the conclusion of Theorem
 \ref{thm4.5} holds for some
\begin{equation}\label{one}
q\in(p,p+\varepsilon).
\end{equation}
So let $\varepsilon\in(0,1)$. By \eqref{4.17}$_1$, there exists $q$
satisfying \eqref{one} such that
\begin{equation}\label{two}
\alpha<\frac{n+2}{2q}\left(1-\frac{1}{\lambda}\right).
\end{equation}
Define $f_0 :\mathbb{R}^n \times\mathbb{R}\to\mathbb{R}$ by
\begin{equation}\label{8.15}
 f_0 (x,t)=\left(\frac{1}{t}\right)^{r}\raisebox{2pt}{$\chi$}_{\Omega_0}(x,t)
\end{equation}
where
\[\Omega_0 =\{(x,t)\in\mathbb{R}^n \times\mathbb{R}:|x|^2 <t<1\}\]
and
\begin{equation}\label{four}
r:=\frac{n+2}{2q}<\frac{n+2}{2p}
\end{equation}
by \eqref{one}.
Then by \eqref{four} and Lemma \ref{lem7.6} we have
\begin{equation}\label{8.16}
 f_0 \in L^p (\mathbb{R}^n \times\mathbb{R})
\end{equation}
and
\begin{equation}\label{8.17}
 J_\alpha f_0 (x,t)\geq C\left(\frac{1}{t}\right)^{r-\alpha}\quad\text{for }(x,t)\in\Omega_0
\end{equation}
where, throughout this entire proof, $C=C(n,\lambda,\alpha,p,q)$ is a
positive constant whose value may change from line to line.

Let $\{T_j \}\subset(0,1/2)$ be a sequence such that
$$T_{j+1}<T_j /4\qquad j=1,2,...$$
and define 
\begin{equation}\label{seven}
 t_j=T_j/2.
\end{equation}
Then
\begin{equation}\label{8.20}
 \Omega_j :=\{(y,s)\in\mathbb{R}^n \times\mathbb{R}:|y|<\sqrt{T_j -s} \text{ and }  t_j <s<T_j \}\subset R_j \subset\Omega_0
\end{equation}
and thus defining
$f_j :\mathbb{R}^n \times\mathbb{R}\to\mathbb{R}$ by 
\begin{equation}\label{nine}
f_j (x,t)=(T_j -t)^{-r}\raisebox{2pt}{$\chi$}_{\Omega_j}(x,t)
\end{equation}
we obtain from \eqref{four} and Lemma \ref{lem7.7} that
\begin{equation}\label{8.21}
 \| f_j \|^{p}_{L^p (\mathbb{R}^n \times\mathbb{R})}
=C(n)\int^{T_j -t_j}_{0}s^{(\frac{n+2}{2p}-r)p-1}ds\
 \to0\quad\text{as }j\to\infty,
\end{equation}
\begin{equation}\label{8.22}
 \| f_j \|_{L^q (R_j )}=\| f_j \|_{L^q (\mathbb{R}^n\times\mathbb{R})}
=\infty\quad\text{for }j=1,2,...,
\end{equation}
and
\begin{equation}\label{8.23}
 J_\alpha f_j (x,t)\geq C\left(\frac{1}{(T_j -t)}\right)^{r-\alpha}
\quad\text{for }(x,t)\in\Omega^{+}_{j}
\end{equation}
where
\[
\Omega^{+}_{j}=\{(x,t)\in\Omega_j :\frac{3T_j}{4}<t<T_j\}.
\]

It follows from \eqref{8.15} and \eqref{8.17} that 
\[
 \frac{f_0 (x,t)}{(J_\alpha f_0 (x,t))^\lambda}\leq
  Ct^{(r-\alpha)\lambda-r} \quad\text{for }(x,t)\in\Omega_0
\]
and from \eqref{two} and \eqref{four} that the exponent
\begin{equation}\label{8.22.5}
(r-\alpha)\lambda-r=\lambda[r(1-1/\lambda)-\alpha]>0.
\end{equation}
Thus
\begin{equation}\label{8.24}
 \sup_{\Omega_0}\frac{f_0}{(J_\alpha f_0)^\lambda} \leq C
\end{equation}
and by \eqref{8.20}
\begin{equation}\label{8.25}
 \sup_{\Omega_j}\frac{f_0}{(J_\alpha f_0)^\lambda}\leq CT_j^{(r-\alpha)\lambda-r}<1
\end{equation}
by taking a subsequence.

By \eqref{nine}, \eqref{8.23}, and  \eqref{8.22.5} we have
\begin{align}\label{8.26}
 \notag \sup_{\Omega^{+}_{j}}\frac{f_j}{(J_\alpha f_j )^\lambda}&\leq C\sup_{(x,t)\in\Omega^{+}_{j}}(T_j -t)^{(r-\alpha)\lambda-r}\\
 &\leq C(T_j -t_j )^{(r-\alpha)\lambda-r}<1
\end{align}
by taking a subsequence.

It follows from \eqref{8.15}, \eqref{nine}, \eqref{8.20}, and
\eqref{seven} that
\begin{equation}\label{8.27}
 \sup_{\Omega_j}\frac{f_0}{f_j}=\sup_{(x,t)\in\Omega_j}
\frac{(T_j -t)^{r}}{t^{r}}\leq\frac{(T_j -t_j )^{r}}{t^{r}_{j}}=1
\end{equation}
and letting $\Omega^{-}_{j}=\Omega_j \backslash\Omega^{+}_{j}$ we see
from \eqref{nine}, \eqref{8.17}, \eqref{8.20}, and \eqref{8.22.5} that
\begin{align}\label{8.28}
 \notag \sup_{\Omega^{-}_{j}}\frac{f_j}{(J_\alpha f_0 )^\lambda}
&\leq
  C\sup_{(x,t)\in\Omega^{-}_{j}}\frac{t^{(r-\alpha)\lambda}}{(T_j
  -t)^{r}}
\leq C\frac{T^{(r-\alpha)\lambda}_{j}}{(T_j/4)^{r}}\\
 &=CT^{(r-\alpha)\lambda-r}_{j}<\frac{1}{2}
\end{align}
by taking a subsequence.

Taking an appropriate subsequence of $f_j$ and letting
$$f=f_0 +\sum^{\infty}_{j=1}f_j$$
we find from \eqref{8.16} and \eqref{8.21} that $f$ satsfies \eqref{4.18}.

In $\Omega^{+}_{j}$ we have by \eqref{8.25} and \eqref{8.26} that
\begin{align*}
 f&=f_0 +f_j \leq(J_\alpha f_0 )^\lambda +(J_\alpha f_j )^\lambda \\
 &\leq(J_\alpha (f_0 +f_j ))^\lambda \leq(J_\alpha f)^\lambda .
\end{align*}
In $\Omega^{-}_{j}$ we have by \eqref{8.27} and \eqref{8.28} that 
$$f=f_0 +f_j \leq2f_j \leq(J_\alpha f_0 )^\lambda \leq(J_\alpha f)^\lambda .$$
In $\Omega_0 \backslash\cup^{\infty}_{j=1}\Omega_j$ we have by \eqref{8.24} that
$$f=f_0 \leq C(J_\alpha f_0 )^\lambda \leq C(J_\alpha f)^\lambda.$$
In $(\mathbb{R}^n \times\mathbb{R})\backslash\Omega_0
,\,f=0\leq(J_\alpha f)^\lambda$.  Thus, after scaling $f$, we see that
$f$ is a solution of \eqref{4.5}, \eqref{4.6}.  Also \eqref{4.19} holds
by \eqref{8.22}. 
\end{proof}

\begin{proof}[Proof of Theorem \ref{thm4.6}]
  By \eqref{4.20}$_1$, there exists a unique number $\gamma\in(0,\frac{n+2}{2p}-\alpha)$ such that
\begin{equation}\label{8.29}
 \lambda=\frac{\frac{n+2}{2p}-\gamma}{\frac{n+2}{2p}-\alpha-\gamma}.
\end{equation}
Let $f_0$ and $\Omega_0$ be as in Lemma \ref{lem7.6}.  Then by \eqref{8.29} and Lemma \ref{lem7.6} we have
\begin{equation}\label{8.30}
 f_0 \in X^p
\end{equation}
and
\begin{equation}\label{8.31}
 f_0 \leq C(J_\alpha f_0 )^\lambda \quad\text{in } \mathbb{R}^n \times\mathbb{R}
\end{equation}
where in this proof $C=C(n,\lambda,\alpha,p)$ is a positive constant whose value may change from line to line.  Let $\{T_j \},\,\{t_j \}\subset(2,\infty)$ satisfy
$$T_{j+1}\geq4T_j  \quad\text{ and }\quad T_j =2t_j$$
and define $f_j :\mathbb{R}^n \times\mathbb{R}\to\mathbb{R}$ by
\begin{equation}\label{8.32}
 f_j (x,t)=\Biggl(\frac{1}{T_j -t}\Biggr)^{\frac{n+2}{2p}-\gamma}\raisebox{2pt}{$\chi$}_{\Omega_j}(x,t)
\end{equation}
where
$$\Omega_j :=\{(x,t)\in\mathbb{R}^n \times(T_j /2,T_j ):|x|<\sqrt{T_j -t}\}.$$
Then
\begin{equation}\label{8.33}
 \Omega_j \subset R_j \subset\Omega_0 ,
\quad\Omega_j \cap\Omega_k =\emptyset\quad\text{for }j\neq k,
\end{equation}
\begin{equation}\label{8.34}
 \text{inf }\{t:(x,t)\in\Omega_j \}=T_j /2\to\infty\quad\text{as }j\to\infty,
\end{equation}
and by \eqref{8.32}, \eqref{8.29}, and Lemma \ref{lem7.7} we have
\begin{equation}\label{8.35}
 f_j \in L^p (\mathbb{R}^n \times\mathbb{R})
\end{equation}
and
$$f_j \leq C(J_\alpha f_j )^\lambda \quad\text{in } \Omega^{+}_{j}$$
where
$$\Omega^{+}_{j}=\{(x,t)\in\Omega_j:\frac{3T_j}{4}<t<T_j \}.$$
It follows therefore from \eqref{8.31} that
\begin{equation}\label{8.36}
 f_0 +f_j \leq C((J_\alpha f_0 )^\lambda +(J_\alpha f_j )^\lambda )\leq C(J_\alpha (f_0 +f_j ))^\lambda \quad\text{in } \Omega^{+}_{j}.
\end{equation}
In $\Omega^{-}_{j}:=\Omega_j \backslash\Omega^{+}_{j}$ we have
$$\frac{f_j}{f_0}=\Biggl(\frac{t}{(T_j -t)}\Biggr)^{\frac{n+2}{2p}-\gamma}\leq\Biggl(\frac{\frac{3}{4}T_j}{\frac{1}{4}T_j}\Biggr)^{\frac{n+2}{2p}-\gamma}=3^{\frac{n+2}{2p}-\gamma}$$
and thus we obtain from \eqref{8.31} that
\begin{equation}\label{8.37}
 f_0 +f_j \leq Cf_0 \leq C(J_\alpha f_0 )^\lambda \leq C(J_0 (f_0 +f_j ))^\lambda \quad\text{in } \Omega^{-}_{j}.
\end{equation}
Let $f=f_0 +\sum^{\infty}_{j=1}f_j$.  Then clearly $f$ satisfies \eqref{4.6} and by \eqref{8.30}, \eqref{8.35}, and \eqref{8.34} we see that $f$ satisfies \eqref{4.21}.

In $\Omega_j$ we have by \eqref{8.33}$_2$, \eqref{8.36}, and \eqref{8.37} that
$$f=f_0 +f_j \leq C(J_\alpha (f_0 +f_j ))^\lambda \leq C(J_\alpha f)^\lambda$$
and in $(\mathbb{R}^n \times\mathbb{R})\backslash\cup^{\infty}_{j=1}\Omega_j$ we have by \eqref{8.31} that
$$f=f_0 \leq C(J_\alpha f_0 )^\lambda \leq C(J_\alpha f)^\lambda .$$
Thus after scaling $f$, we find that $f$ satisfies \eqref{4.5}.

Since $|R_j |<\infty$, we can for the proof of \eqref{4.22} assume instead of \eqref{4.20}$_2$ that
$$q=\frac{n+2}{2\alpha}(1-\frac{1}{\lambda})$$
and hence by \eqref{8.29} we get
$$\frac{n+2}{2p}-\gamma=\frac{\alpha}{1-\frac{1}{\lambda}}=\frac{n+2}{2q}.$$
Consequently from \eqref{8.33}$_1$, \eqref{8.32}, and Lemma \ref{lem7.7} we find that
$$\| f\|_{L^q (R_j )}\geq\| f_j \|_{L^q (\Omega_j )}=\infty \quad\text{for }j=1,2,...$$
which proves \eqref{4.22} 
\end{proof}

\appendix
\section{Appendix}
For the proof of Theorem \ref{thm2.3}(ii) we will need the following
result due to Nogin and Rubin \cite{NR} concerning the inversion of the
operator $J_\alpha$ in the framework of the spaces
$L^p (\mathbb{R}^n \times\mathbb{R})$.  See also \cite[Theorem 9.24]{SK}.

\begin{thm}\label{thmA.1}
 Suppose $0<\alpha<\frac{n+2}{2p},\,1<p<\infty$, and $u=J_\alpha f$ with $f\in L^p (\mathbb{R}^n \times\mathbb{R})$.  Then
 $$\lim_{\varepsilon\to 0^{+}}J^{-\alpha}_{\varepsilon}u=f\quad\text{in } L^p (\mathbb{R}^n \times\mathbb{R})$$
 where
 \begin{equation}\label{A.1}
  J^{-\alpha}_{\varepsilon}u(x,t)=C(n,\alpha,l)\iint_{\mathbb{R}^n \times(\varepsilon,\infty)}\frac{(\Delta^{l}_{y,\tau}u)(x,t)}{\tau^{1+\alpha}}e^{-\frac{|y|^2}{4}}\, dy \, d\tau
 \end{equation}
 and
 \begin{equation}\label{A.2}
  (\Delta^{l}_{y,\tau}u)(x,t)
=\sum^{l}_{k=0}(-1)^k \binom{l}{k}u(x-y\sqrt{k\tau},t-k\tau),\quad l>\alpha.
 \end{equation}
\end{thm}

\noindent {\bf Acknowledgments}

\noindent The author thanks the anonymous referee for very helpful comments.


\begin{thebibliography}{33}

\bibitem{AABP} B. Abdellaoui, A. Attar, R. Bentifour, I. Peral, On
  fractional p-Laplacian parabolic problem with general data,
  Ann. Mat. Pura Appl. (4) 197 (2018) 329--356.

\bibitem{AV} E. Affili, E. Valdinoci, Decay estimates for evolution
  equations with classical and fractional time-derivatives,
  J. Differential Equations, https://doi.org/10.1016/j.jde.2018.09.031

\bibitem{AMPP} Boumediene Abdellaoui, Maria Medina, Ireneo Peral,
  Ana Primo, Optimal results for the fractional heat equation
  involving the Hardy potential, Nonlinear Anal. 140 (2016) 166--207.

\bibitem{A} Mark Allen, A nondivergence parabolic problem with a
  fractional time derivative, Differential Integral Equations 31
  (2018) 215--230.

\bibitem{ACV} Mark Allen,  Luis Caffarelli,  Alexis Vasseur,  A
  parabolic problem with a fractional time
  derivative, Arch. Ration. Mech. Anal. 221 (2016) 603--630.

\bibitem{ACM} Ioannis Athanasopoulos, Luis Caffarelli, Emmanouil Milakis,
  On the regularity of the non-dynamic parabolic fractional
  obstacle problem, J. Differential Equations 265 (2018) 2614--2647.

\bibitem{BV} Matteo Bonforte, Juan Luis V\'azquez,  A priori estimates
  for fractional nonlinear degenerate diffusion equations on bounded
  domains, Arch. Ration. Mech. Anal. 218 (2015) 317--362.

\bibitem{CVW} Huyuan Chen,  Laurent V\'eron, Ying Wang, Fractional
  heat equations with subcritical absorption having a measure as
  initial data, Nonlinear Anal. 137 (2016) 306--337.

\bibitem{DS} Mat\'ias G. Delgadino,  Scott Smith,  H\"older estimates for
  fractional parabolic equations with critical divergence free drifts,
  Ann. Inst. H. Poincaré Anal. Non Linéaire 35 (2018) 577--604.

\bibitem{DVV} S. Dipierro, E. Valdinoci, V. Vespri, Decay estimates
  for evolutionary equations with fractional time-diffusion,
  J. Evol. Equ.  https://doi.org/10.1007/s00028-019-00482-z

\bibitem{FKRT} Giulia Furioli, Tatsuki Kawakami, Bernhard Ruf, 
  Elide Terraneo, Asymptotic behavior and decay estimates of the
  solutions for a nonlinear parabolic equation with exponential
  nonlinearity, J. Differential Equations 262 (2017) 145--180.

\bibitem{GW} Ciprian G. Gal,  Mahamadi Warma,  On some degenerate
  non-local parabolic equation associated with the fractional
  p-Laplacian, Dyn. Partial Differ. Equ. 14 (2017) 47--77.

\bibitem{GR} V. R. Gopala Rao, A characterization of parabolic
  function spaces, Amer. J. Math. 99 (1977) 985--993.

\bibitem{JS}  Mohamed Jleli,  Bessem Samet, The decay of mass for a
  nonlinear fractional reaction-diffusion equation, Math. Methods
  Appl. Sci. 38 (2015) 1369--1378.

\bibitem{K}  Jan Kadlec,  Solution of the first boundary value problem
  for a generalization of the heat equation in classes of functions
  possessing a fractional derivative with respect to the
  time-variable, (Russian) Czechoslovak Math. J. 16 (91) (1966)
  91--113.

\bibitem{KSVZ} Jukka Kemppainen, Juhana Siljander, Vicente Vergara,
  Rico Zacher, Decay estimates for time-fractional and other non-local
  in time subdiffusion equations in $\mathbb{R}^d$. Math. Ann. 366
  (2016) 941--979

\bibitem{M} M. Mirzazadeh, Analytical study of solitons to nonlinear
  time fractional parabolic equations. Nonlinear Dynam. 85 (2016)
  2569--2576.

\bibitem{MT} Luc Molinet, Slim Tayachi, Remarks on the Cauchy problem
  for the one-dimensional quadratic (fractional) heat equation,
  J. Funct. Anal. 269 (2015) 2305--2327.

\bibitem{NR} V. A. Nogin, B. S. Rubin,  The spaces
  $L_{p,r}^\alpha(R^{n+1})$ of parabolic potentials, (Russian)
  Anal. Math. 13 (1987) 321--338.

\bibitem{NS} K. Nystr\"om,   O. Sande, Extension properties and boundary
  estimates for a fractional heat operator, Nonlinear Anal. 140
  (2016) 29--37.

\bibitem{OD}  Ebru Ozbilge,  Ali Demir,  Identification of unknown
  coefficient in time fractional parabolic equation with mixed
  boundary conditions via semigroup approach, Dynam. Systems Appl. 24
  (2015) 341--348.

\bibitem{PV} Fabio Punzo,  Enrico Valdinoci, Uniqueness in weighted
  Lebesgue spaces for a class of fractional parabolic and elliptic
  equations, J. Differential Equations 258 (2015) 555--587.

\bibitem{R} Walter Rudin, Functional analysis, McGraw-Hill Series in
  Higher Mathematics, McGraw-Hill Book Co., New
  York-D\"usseldorf-Johannesburg, 1973.


\bibitem{SK} Stefan G. Samko,  Hypersingular integrals and their
  applications, Analytical Methods and Special Functions, 5, Taylor \&
  Francis, Ltd., London, 2002.

\bibitem{SP} Charles H. Sampson,  A Characterization of Parabolic
  Lebesgue Spaces, Dissertation, Rice Univ., 1968.

\bibitem{S} Elias M. Stein, Singular integrals and differentiability
  properties of functions, Princeton Mathematical Series, No. 30,
  Princeton University Press, Princeton, N.J. 1970.

\bibitem{ST} Pablo Ra\'ul Stinga,  Jos\'e L. Torrea, Regularity theory and
  extension problem for fractional nonlocal parabolic equations and
  the master equation, SIAM J. Math. Anal. 49 (2017) 3893--3924.

\bibitem{SS} Fuqin Sun, Peihu Shi,  Global existence and non-existence
  for a higher-order parabolic equation with time-fractional term,
  Nonlinear Anal. 75 (2012) 4145--4155.

\bibitem{V} Vladimir Varlamov,  Long-time asymptotics for the
  nonlinear heat equation with a fractional Laplacian in a ball,
  Studia Math. 142 (2000) 71--99.

\bibitem{VV} Juan Luis V\'azquez, Bruno Volzone,  Symmetrization for
  linear and nonlinear fractional parabolic equations of porous medium
  type, J. Math. Pures Appl. (9) 101 (2014) 553--582.

\bibitem{VPQR} Juan Luis V\'azquez,  Arturo de Pablo, Fernando Quir\'os, 
  Ana Rodr\'iguez, Classical solutions and higher regularity for
  nonlinear fractional diffusion equations, J. Eur. Math. Soc. 19
  (2017) 1949--1975.

\bibitem{VZ} Vicente Vergara, Rico Zacher, Optimal decay estimates
  for time-fractional and other nonlocal subdiffusion equations via
  energy methods, SIAM J. Math. Anal. 47 (2015) 210--239.

\bibitem{ZS} Quan-Guo Zhang,  Hong-Rui Sun,  The blow-up and global
  existence of solutions of Cauchy problems for a time fractional
  diffusion equation, Topol. Methods Nonlinear Anal. 46 (2015)
  69--92.

\end{thebibliography}
\end{document}